\documentclass[11pt]{amsart}

\usepackage{amsmath,amsthm,amssymb,amsfonts,amscd,
epsfig,latexsym,graphicx,eucal,marvosym,textcomp,turnstile,extarrows}
 
\usepackage{hyperref}
\newtheorem{theorem}{Theorem}[section]
\newtheorem{lemma}[theorem]{Lemma}
\newtheorem{corollary}[theorem]{Corollary}
\newtheorem{proposition}[theorem]{Proposition}
\newtheorem{remark}[theorem]{Remark}

\theoremstyle{definition}
\newtheorem{definition}[theorem]{Definition}

\numberwithin{equation}{section}
\newcommand{\Tr}{\mathrm{Tr}}

\newcommand{\Ad}{\mathrm{Ad}}
\newcommand{\ad}{\mathrm{ad}}
\newcommand{\PW}{\mathrm{PW}}
\newcommand{\R}{\mathbb{R}}
\newcommand{\C}{\mathbb{C}}
\newcommand{\T}{\mathbb{T}}
\newcommand{\Z}{\mathbb{Z}}
\newcommand{\D}{\mathbb{D}}
\newcommand{\gL}{\Lambda} 
\newcommand{\supp}{\mathrm{supp}}
\newcommand{\eps}{\varepsilon}
\newcommand{\Exp}{\mathrm{Exp}}
\newcommand{\g}[1]{{\mathfrak{#1}}}
\newcommand{\im}{\mathrm{Im}}
\newcommand{\re}{\mathrm{Re}}

\newcommand{\wh}[1]{\widehat{#1}}
\newcommand{\wt}[1]{\widetilde{#1}}

\newcommand{\SO}{\mathrm{SO}}
\newcommand{\GL}{\mathrm{GL}}
\newcommand{\mS}{\mathcal{S}}

\newcommand{\dotcup}{\ensuremath{\mathaccent\cdot\cup}}
\newcommand{\bX}{\mathbf{X}}
\newcommand{\bY}{\mathbf{Y}}
\newcommand{\cL}{\mathcal{L}}
\newcommand{\cM}{\mathcal{M}}

\newcommand{\fa}{\mathfrak{a}}
\newcommand{\fb}{\mathfrak{b}}
\newcommand{\fg}{\mathfrak{g}}
\newcommand{\fh}{\mathfrak{h}}
\newcommand{\fk}{\mathfrak{k}}
\newcommand{\fn}{\mathfrak{n}}

\newcommand{\fq}{\mathfrak{q}}
\newcommand{\fs}{\mathfrak{s}}
\newcommand{\fu}{\mathfrak{u}}

\newcommand{\cD}{\mathcal{D}}

\newcommand{\cO}{\mathcal{O}}
\newcommand{\cS}{\mathcal{S}}

\def\sideremark#1{\ifvmode\leavevmode\fi\vadjust{\vbox to0pt{\vss
 \hbox to 0pt{\hskip\hsize\hskip1em
\vbox{\hsize2cm\tiny\raggedright\pretolerance10000 
 \noindent #1\hfill}\hss}\vbox to8pt{\vfil}\vss}}} 

\newcommand{\ipl}[2]{( #1,#2 )}
\newcommand{\ip}[2]{\langle #1,#2\rangle}
\begin{document}

\title[Paley-Wiener Theorem for Line Bundles]{Paley-Wiener Theorem for Line Bundles over Compact Symmetric Spaces
and New Estimates for the Heckman-Opdam Hypergeometric Functions}

\author{Vivian M. Ho}
\address{Department of Mathematics \\
Louisiana State University\\
Baton Rouge, Louisiana}
\email{vivian@math.lsu.edu}

\author{Gestur \'{O}lafsson}
\address{Department of Mathematics \\
Louisiana State University\\
Baton Rouge, Louisiana}
\email{olafsson@math.lsu.edu}
\begin{thanks}
{The research of G. \'Olafsson was supported by NSF grant DMS-1101337.} 
\end{thanks}

\date{}

\begin{abstract}
Paley-Wiener type theorems describe the image of a given space of
functions, often compactly supported functions, under an integral transform, usually a Fourier transform on
a group or homogeneous space. In this article we proved a Paley-Wiener theorem for smooth sections $f$ 
of homogeneous line bundles on a compact Riemannian symmetric space $U/K$. It characterizes $f$ 
with small support in terms of holomorphic extendability and exponential growth of their $\chi$-spherical Fourier transforms,
where $\chi$ is a character of $K$.

An important tool in our proof is a generalization of Opdam's estimate
for the hypergeometric functions associated to multiplicity functions that are not
necessarily positive. At the same time the radius of the domain where this estimate is valid is increased.
This is done in an appendix. 
\end{abstract}

\maketitle

\section{Introduction}
\label{introduction}

\subsection{Paley-Wiener Theorem}

\noindent
The classical Paley-Wiener theorem identifies the
space of smooth compactly supported functions on $\R^n$ with certain classes of holomorphic functions on $\C^n$ of
exponential growth  via the  Fourier transform on $\R^n$. The exponent is determined by the size of the support, 
see \cite[Theorem 7.3.1]{hoe90}.

There are several generalizations of this theorem to settings where $\R^n$ is replaced by a Lie group or a 
homogeneous space. Of all of those generalizations, the
Riemannian symmetric spaces are best understood, in particular, those  of the noncompact type due to
the work of Helgason \cite{h5, h6} and Gangolli
\cite{ga} for smooth functions, and Eguchi, Hashizume and Okamoto \cite{eho} and Dadok \cite{da} for 
distributions. The case of semisimple Lie groups was due to Arthur \cite{ar}, the case
of general reductive symmetric spaces was done by van den Ban and Schlichtkrull \cite{vs, vs1}, and the case of 
hyperbolic spaces was treated by Andersen \cite{an}. More recently the Paley-Wiener type theorems have been extended to 
the case of the Heckman-Opdam hypergeometric Fourier transform and
the compact settings, \cite{bop,bop2005b,go,op2004a,op2004b,op,os,os1,os2}
and even infinite dimensional Lie groups \cite{ow}. We refer to \cite{do2011} and references therein 
for overview and further discussions.
In the compact case every smooth function
has a compact support, so the Paley-Wiener theorem is a local
statement and   only
valid for functions supported in sufficiently small balls around the base point.

The first part of this article is  inspired by  \cite{os} which deals with functions
on compact symmetric spaces. We aim to generalize the results of \cite{os} to line bundles over
symmetric spaces of compact type. Similar results of noncompact type was treated in \cite{sh}.

Consider a Riemannian symmetric space $\bY =U/K$ of compact type. 
Let $\bX=G/K_o$ be the noncompact dual
symmetric space, $K_o$ the connected component
of $K$ containing the identity element (note that $K$ is connected
if $U$ is assumed to be simply connected). Let $\chi$ be a character of $K$. 
A homogeneous line bundle over $\bY$ is then given by the fiber-product
$U\times_\chi \C =: \cL_\chi$. The space of smooth sections on $\cL_\chi$ is isomorphic to the space
$C^\infty\, (\bY;\cL_\chi)$ of all smooth functions $f: U\to \C$ such that
\begin{equation}\label{eq:sections}
f(uk)=\chi (k)^{-1}f (u)\quad\text{ for all }\,
k\in K\text{ and  all } u\in U\, .
\end{equation}
Similarly, one defines the space of $L^2$-sections, $L^2(\bY;\cL_\chi)$. 
There is a natural unitary representation of $U$ on $L^2(\bY;\cL_\chi)$, the \textit{regular representation},
given by left translation
\[\lambda (g)f(u)=f(g^{-1}u)\, .\]
An irreducible representation of $U$ is said to be $\chi$-spherical if  there exists a nonzero vector
$e_\mu \in V_\mu$ such that
\begin{equation}\label{eq:chi-spher}
\pi_\mu (k)e_\mu =\chi (k)e_\mu\quad \text{ for all } k\in K\, .
\end{equation} 
As $U$ is compact the regular representation decomposes into direct sum of irreducible representations.
The representations that are contained in this sum are the
$\chi$-spherical representations and each is contained with multiplicity one, see \cite{sc}. The classification
of the $\chi$-spherical representations in terms of their highest weights can also be found in
\cite{sc}. In this article we give a slightly simpler description of this set, see  Proposition 
\ref{pro2}.

The replacement for the $K$-biinvariant functions studied in \cite{os} are
the $\chi$ bicovariant functions 
\begin{equation}\label{eq:shper} f(k_1uk_2)
 =\chi (k_1k_2)^{-1} f(u) \quad\text{ for all }\,
k_1,k_2\in K\text{ and  all } u\in U\, .
\end{equation}
This space is denoted by $C^\infty (U//K; \cL_\chi)$. The replacement for the spherical functions on $\bY$ are the
$\chi$-spherical functions $\psi_\mu$ given by
\begin{equation}\label{eq:spher}
\psi_\mu (u)=\ip{e_\mu}{\pi_\mu (u)e_\mu}
\end{equation}
and the $\chi$-spherical Fourier transform is given by
\begin{equation}\label{eq:FT}
\widehat{f}(\mu )=\ip{f}{\psi_\mu}=\int_U f(u)\psi_\mu (u^{-1})\, du\, .
\end{equation}

We investigate $\chi$-bicovariant sections $f$ supported in a closed metric ball $\overline{B}_r(x_o)$ of sufficiently 
small radius $r>0$ around a base point $x_o$. 
We show the $\chi$-spherical Fourier transform of $f$ extends to a holomorphic function of exponential type $r$ (it 
is exactly the size of the support of $f$). The image of $\chi$-spherical Fourier transform is certain space of holomorphic functions, 
see the Paley-Wiener theorem (Theorem \ref{thm1}). 



The proof of the Paley-Wiener theorem uses new estimates for the Heckman-Opdam hypergeometric functions.
This is the main content of the second part of this article. 


\subsection{The Hypergeometric Functions}
The Heckman-Opdam hypergeometric functions were introduced by a series of joint work of Heckman and Opdam \cite{hop, he, op1, op2}.
They are joint eigenfunctions of a commuting algebra of differential operators associated to
a root system and a multiplicity parameter (which is a Weyl group invariant function 
on the root system). The multiplicities can be arbitrary complex numbers. 
The hypergeometric functions are holomorphic, Weyl group invariant, and normalized by
 the value one at the identity.
When the root multiplicities do correspond to those of a Riemannian symmetric space, the hypergeometric functions are
simply the restrictions  of spherical functions to a Cartan subspace. 
We refer to \cite[Part I, Chapter $4$]{hs} and \cite[Section $3.2$]{sh} for an introduction to
the Harish-Chandra asymptotic expansion for the $\chi$-spherical functions on $G$.

The $\chi$-spherical functions on $G$ were identified as hypergeometric functions in \cite{hs}, but whose
multiplicity parameters are not necessarily positive. 
So far uniform exponential estimates on the growth
behavior of hypergeometric functions have only been proven with the condition that all multiplicities are positive \cite{op},
and hence not applicable in our situation. We therefore extend such results. Our estimate 
\begin{enumerate}
\item works for the hypergeometric functions whose multiplicity parameters are allowed to be certain 
negative numbers
\item the domain of the hypergeometric functions where this new estimate is valid is increased (
double the size of the original domain in \cite{op}).
\end{enumerate}
This is of crucial importance to obtain the right exponential type growth of $f$ under the $\chi$-spherical Fourier
transform (see the proof of part (1) of Theorem \ref{thm1}). The new estimate is proved in Proposition \ref{hypfun:estim}, Appendix \ref{hyper}.
The technique to prove it is inspired by ideas from \cite{op}. 

An alternative method to attack this problem is to apply some suitable shift operators 
in order to move negative multiplicities to positive ones. 
This allows us to use many well-known results of the hypergeometric functions associated with positive multiplicities. 
This method was investigated in \cite{ho}.

\subsection{Plan of the Article}
In Section \ref{preliminaries} we introduce basic notations and structure theory on
Riemannian symmetric spaces. In Section \ref{Fourier} and Section \ref{s:spheri} we discuss harmonic analysis related to line bundles over compact symmetric
spaces, including the theory of highest weights for $\chi$-spherical representations, elementary spherical functions 
of type $\chi$, and $\chi$-spherical Fourier transform. In Section \ref{s:PW} we define the relevant Paley-Wiener space and state the
main theorem (Theorem \ref{thm1}), to prove which we need some tools of differential operators (Sections \ref{s:holExt}) and
hypergeometric functions (Appendix \ref{hyper}). Sections \ref{s:PW} and \ref{s:bij} contains the main body of the proof.
In Section \ref{s:PW} we show the $\chi$-spherical Fourier transform maps into the Paley-Wiener space. In Section \ref{s:surj} we prove
the bijection of the $\chi$-spherical Fourier transform. The proof of injectivity is obvious, 
while the proof of surjectivity is by reduction to
the Paley-Wiener theorem for the group case, recalled in Section \ref{s:pwtgp}.

\tableofcontents 

\section{Notation and Preliminaries}
\label{preliminaries}
\noindent
The material in this section is standard. We refer to  \cite{h1} for references. We will
often need \cite{h2,h3} too. We use the notation from the introduction mostly
without reference.

\subsection{Symmetric spaces}

We recall some standard notations and facts related to symmetric spaces.

A Riemannian symmetric space of the compact type can be realized as $\bY = U / K$ where
$U$ is a connected semisimple compact Lie group and $K \subseteq U$ a closed symmetric subgroup.
Thus, there exists a nontrivial involution $\theta: U \to U$ such that
\[U_o^\theta \subseteq K \subseteq U^\theta=\{u\in U\mid \theta (u)=u\}\, .\] 
We fix the base point $x_o=eK$ and write $a\cdot (bK)=(ab)K$ for the action of $U$ on $\bY$.  
Assume $U$ is simply connected. Then $\bY$ is simply connected.
It is not necessary to assume that $\bY$ is simply connected, 
but it makes several arguments simpler. In particular, the classification of the $\chi$-spherical representation is simpler. 
Note that the spherical harmonic analysis on a general compact symmetric space can be reduced to the simply 
connected case (see \cite[p.4860]{opa} and \cite{os}).

Let $\g{u}$ be the Lie algebra of $U$. Then $\theta$ induces an involution on $\g{u}$, also denoted by $\theta$. 
Decompose $\g{u} = \g{k} \oplus \g{q}$ into $\pm 1$-eigenspaces of $\theta$, where $\g{k}=\g{u}^\theta$ is the
Lie algebra of $K$. We identify $\fq$ with the tangent space $T_{x_o}(\bY )$ of $\bY$ at $x_o$ by
\[D_Xf(x_o)=\left.\dfrac{d}{dt}\right|_{t=0}f(\Exp (tX) )\]
where $\Exp (Y)=\exp (Y)\cdot x_o$. Then $T\bY \cong \fq\times_K U$. 
Any positive $K$-invariant bilinear form on $\fq$ defines a
Riemannian metric on $\bY$. As an example, we let $\ip{\,\cdot\,}{\,\cdot\,}$ be the inner product on $\fu$ defined by
\[\ip{ X}{ Y } = - \Tr\, ( \ad(X)\ad(Y))\, .\] 
This inner product is $K$-invariant and defines a Riemannian metric on $\bY$.  
The inner product on $\g{u}$ gives an inner product on the
dual space $\g{u}^\ast$ in a canonical way, and by hermitian extension they induce
$U$-invariant inner products on 
$\g{u}_\C = \g{u} \otimes_\R \C $ and the complex dual space $\g{u}_\C^\ast$, all denoted by the
same symbol. We write $\|\lambda \|= \sqrt{\ip{\lambda}{\lambda}}$ for the corresponding norm.
Similar notations will be used for other Lie algebras and vector spaces.
The $\C$-bilinear extension to $\g{u}^\ast_\C$ will be denoted by
$\lambda ,\mu\mapsto \ipl{\lambda}{\mu}$. 

A maximal abelian subspace of $\fq$ is called a Cartan subspace for $\bY$ (or $(U,K)$).
All Cartan subspaces are $K$-conjugate and  their common dimension is
called  the rank of $\bY$.
>From now on we fix a Cartan subspace $\fb$. Let $n=\dim \fb$.  We fix a  Cartan subalgebra
$\g{h}$ of $\g{u}$ containing $\g{b}$. Then $\g{h}$ is $\theta$-stable and 
\begin{equation}
\label{eq20}
\g{h} = (\g{h} \cap \g{k}) \oplus \g{b}\, .
\end{equation}
Let $B=\exp (\g{b})$ be the analytic subgroup of $U$ with Lie algebra $\g{b}$.   
The subspace $B\cdot x_o\simeq B/B\cap K$ is a Cartan subgroup of $\bY$. Note that
$B\cap K$ is finite.

Since $U$ is compact, it admits a finite dimensional faithful unitary representation. Thus
$U\subset \mathrm{U}(p)\subset \GL (p,\C)$
for some $p$. As $\fu \subset \fu (p)$ it follows that $\fu \cap i \fu=\{0\}$ and
hence $\fu_\C\simeq \fu \oplus i\fu$. Let $U_\C$ denote the analytic subgroup
of $\GL (p,\C)$ with Lie algebra $\fu_\C$. Note that $U_\C$ is simply
connected. Let $\fs = i\fq$ and let $\fg=\fk\oplus \fs$. Note
that $\fg$ is a Lie algebra.  Denote
by $G$ the analytic subgroup of $\GL (p,\C)$ (and $U_\C$) with Lie algebra
$\g{g}$. Then $G=K\exp \g{s}$ and $\bX=G/K$ is a Riemannian symmetric space of the
noncompact type. It is called the Riemannian dual of $\bY$. The Riemannian structure on $\bX$ is again determined
by the inner product $\ip{X}{Y}=\Tr \, \ad (X)\ad (Y)$ on $\g{s}$.

We will from  now on view  $\bX$ and $\bY$ as real forms of the complex
homogeneous space $U_\C/K_\C$ where $K_\C=\exp \fk_\C\subset U_\C$. Again
$x_o=eK_\C$ is the common base point.

The involution
$\theta$ extends to a holomorphic involution on $U_\C$, also denoted by $\theta$. We
also write $\theta$ for the restriction to $G$ and note that $\theta$   is a Cartan
involution on $G$.   

Again, a maximal subspace of $\fs$ is called a Cartan subspace of $\bX$. The Cartan subspaces of $\bX$ are conjugate under $K$. We fix from now on the Cartan subspace
$\fa=i\fb$ where $\fb$ is the fixed Cartan subspace for $\bY$. Then
$A\simeq \g{a}$ is simply connected.  We have $\g{a}_\C =\g{b}_\C = \g{a}\oplus \g{b}$ and
$A_\C=\exp (\g{a}_\C) =AB$. We denote by $\log : A_\C\to \g{a}_\C$
the multivariable inverse of $\exp|_{\g{a}_\C}$.  It is a single valued isomorphism $\log :A \to \g{a}$.

\subsection{Line bundles on hermitian symmetric spaces}

We will from now on assume that $U$ is simple and that there exists nontrivial line bundles over $\bY$
although some of our statements are true in general.
Let $\chi : K\to \mathbb{T}$ be a character. Then the homogeneous line bundle $\cL_\chi\to \bY$ is defined as
\[\cL_\chi=U \times_\chi \C = (U \times \C) / K\]
where $\C$ denotes the complex numbers  with the
action $k\cdot z=\chi (k)^{-1}\, z$ and $K$ acts on $U\times \C_\chi$ by $(u,z)\cdot k=(uk,k\cdot z)$.
The existence of a homogeneous line bundle over $\bY$ is equivalent to the
existence of a nontrivial character which in turn is equivalent to $\dim Z(K)= 1$, 
where $Z(K)$ is the center of $K$. Let $\g{z}$ denote the center of $\g{k}$ and $\g{k}_1=[\g{k}, \g{k}]$. Then
$\dim \g{z}=1$ and $\g{k}=\g{z}\oplus \g{k}_1$. Let $K_1$ denote the analytic subgroup of $K$ (and $U$) with
Lie algebra $\g{k}_1$. Then $K_1$ is closed and $K=Z(K)_oK_1$. The spaces $\bX$ and $\bY$ are
hermitian symmetric spaces of the noncompact and compact type respectively. They are complex homogeneous
spaces where the complex structure is given by the adjoint action of a central element in $\g{k}$ with
eigenvalues $0$ and $\pm i$.

Up to coverings the irreducible spaces with $\dim \g{z}=1$ are given in the following table  (cf. \cite[p.516, 518]{h1}). 
Here $n=\dim \fb=\dim \fa$ is the
rank of $\bX$ and $\bY$ and $d = \dim_\R \bY=\dim_\R \bX$. The conditions listed in the last column
are given to prevent coincidence between different classes due to lower dimensional isomorphisms.

{\small
\begin{equation}\label{t1}
\centering
\begin{tabular}[c]{|c||c|c|c|c|c|c|c|}
\hline
\multicolumn{8}{|c|}{}\\
\multicolumn{8}{|c|}{The Hermitian symmetric spaces}\\
\multicolumn{8}{|c|}{}\\
\hline\hline
& class & $G$ & $U$ & $K$ & $n$ & $d$ & \\
\hline
\hline
$1$ & $A III$ & $\mathrm{SU}\, (p,\, q)$ & $\mathrm{SU}\, (p + q)$ & $\mathrm{S}\, (\mathrm{U} (p) \times \mathrm{U} (q) )$ &  $p$ & 
$2 p q$ & $q \geq p \geq 1$\\
\hline
$2$ & $BD I$ & $\SO_o\, (p,\, q)$ & $\SO\, (p+q)$ & $\SO\, (p) \times \SO\; (q)$ & $p$ & $p q$
& $p=2$, $q \geq 5$ \\
\hline
$3$ & $D III$ & $\SO^\ast\, (2\, j)$ & $\SO\, (2\, j)$ &  $\mathrm{U}\, (j)$ & $[\frac{1}{2}\, j]$ & $j (j-1)$ & $j \geq 5$\\
\hline
$4$ & $C I$ & $\mathrm{Sp}\, (j,\, \R)$ & $\mathrm{Sp}\, (j)$ & $\mathrm{U}\, (j)$ & $j$ & $j (j + 1)$ & $j \geq 2$ \\
\hline
$5$ & $E III$ & $\g{e}_{6\, (- 14)}$ & $\g{e}_{6\, (- 78)}$ & $\g{s o}\, (10) + \R$ & $2$ & $32$ &\\
\hline
$6$ & $E VII$ & $\g{e}_{7\, (- 25)}$ & $\g{e}_{7\, (- 133)}$ & $\g{e}_6 + \R$ & $3$ & $54$ &\\
\hline
\end{tabular}
\end{equation}
}

\begin{remark}{\rm 
 In Case $1$, 
$$\bY = \mathrm{SU}\, (p + q) / \mathrm{S}\, (\mathrm{U} (p) \times \mathrm{U} (q))$$ 
is the complex Grassmann manifold of $p$-dimensional subspaces in $\C^{p+q}$. In Case $2$, 
$$\bY = \SO\, (p+q) / (\SO\, (p) \times \SO\, (q))$$ 
is a covering of $\SO\, (p+q) / \mathrm{S}\, (\mathrm{O}\, (p) \times \mathrm{O}\, (q))$, 
the real Grassmann manifold of $p$-dimensional subspaces in $\R^{p+q}$. }
\end{remark}

We recall the parametrization of the group of characters given in \cite{sc}. Fix
$Z\in \g{z}$ as in \cite[p.283, (3.1)]{sc}. Thus  $\exp (t Z) \in Z (K)$ for all $t\in \R$, and
$\exp (t Z) \in K_1$ if and only if $t \in 2 \pi \Z$. 

\begin{proposition}[H. Schlichtkrull]
\label{pro0}
Let $l \in \Z$. Define $\chi_l: K \to \mathbb{T}$ by
\[ \chi_l (\exp (tZ)k)=e^{ilt}\, ,\; t\in \R \text{ and } k\in K_1\, .\] 
Then $\chi_l$ is a well defined character on $K$. If $\chi$ is a character on  $K$, then there is a unique $l \in \Z$ such that 
$\chi = \chi_l$.
\end{proposition}

\begin{proof}
See Proposition 3.4 in \cite{sc} and its following comment.
\end{proof}

Since all one dimensional representations $\chi$ of $K$ have this form, hereafter, we parametrize $\chi = \chi_l$ for
$l \in \Z$. If $l = 0$, then $\chi_0$ is trivial.

\subsection{Root structures and the Weyl group}  
For $\alpha \in \g{h}_{\C}^\ast$ let
$$\g{u}_{\C,\, \alpha} = \{X \in \g{u}_\C\; |\; (\forall\, H \in \g{h}_\C)\; [H,\, X] = \alpha\, (H)\, X\}.$$
If $\alpha \ne 0$ and $\g{u}_{\C,\, \alpha}\not=\{0\}$ then $\alpha $ is said to be a root. We write  $\Delta = \Delta\, (\g{u},\, \g{h})$
for the set of roots. Similarly we define $\g{u}_{\C,\, \beta}$ for $\beta \in \g{b}_\C^\ast$ and
write $\Sigma=\Sigma (\g{u},\g{b})$ for the set of (restricted) roots. Note that 
$\Delta \subset i \g{h}^\ast$, $\Sigma\subset i\fb^*$, and $\Sigma = \Delta|_{\fb}\setminus \{0\}$. 
The numbers  
\[m_\beta=\dim_\C \fu_{\C,\beta}=\# \{\alpha \in \Delta \mid \alpha|_{\fb}=\beta\},\quad \beta \in \Sigma\]
are called multiplicities. Also note that $\fu_{\C,\alpha}\cap \fu=\{0\}$ for all $\alpha \in \Delta\cup \Sigma$.

Similarly, we can define the roots of $\g{a}$ in $\g{g}$ and we have
$\Sigma = \Sigma (\g{g},\g{a})$. We have $\fu_{\C,\beta}=\fg_{\beta}\oplus i\fg_{\beta}$ and $m_\beta=\dim_\R \fg_\beta$ for all $\beta \in \Sigma$. Working with roots it is therefore more convenient
to work with $\fg$ and $\fa$ rather than the pair $\fu$ and $\fb$.

An element $X \in \g{a}$ is called regular if $\alpha (X) \ne 0$ for all $\alpha \in \Sigma$. 
The subset $\g{a}^{\mathrm{reg}} \subset \g{a}$ is dense and is a finite union of open cones called
Weyl chambers. We fix a Weyl chamber $\g{a}^+$ and let
\[\Sigma^+ :=\{\alpha \in\Sigma  \mid (\forall H\in \fa^+)\,\, \alpha (H)>0\}\,. \]
We choose a positive system $\Delta^+$ in $\Delta$ 
such that if $\alpha\in\Delta^+$ and $\alpha|_{\fa}\not= 0$ then $\alpha|_{\fa}\in \Sigma^+$. 
Let $\Delta_0 = \{\alpha \in \Delta\, \mid\, \alpha |_\fa = 0\}$ and $\Delta_0^+ = \Delta_0 \cap \Delta^+$. Let
\begin{equation}\label{eq:rho}
\rho = \frac{1}{2}\sum_{\alpha\in\Sigma^+} m_\alpha \alpha \quad \text{ and }\quad
\rho_\fh =\frac{1}{2}\sum_{\beta\in\Delta^+} \beta\, .
\end{equation}
Let $\rho_0=\sum_{\alpha\in\Delta_0^+}\alpha$. Then $\rho_0\in i(\fh\cap \fk)^*$ and
\begin{equation}\label{eq:resRho}
\rho_\fh |_{\fa}=\rho\quad \text{ and }\quad \rho_\fh=\rho +\rho_0\, .
\end{equation}

If $\alpha \in \Sigma$ then it can happen that either $\alpha / 2 \in \Sigma$ or $2\alpha \in \Sigma$, but not both.
A root $\alpha
\in \Sigma$ is said to be unmultiplicable if $2 \alpha \notin \Sigma$ and indivisible if $\alpha / 2 \notin \Sigma$. Denote by
\begin{equation}\label{eqSast}
\Sigma_\ast = \{\alpha \in \Sigma \mid 2 \alpha \notin \Sigma\}\,\quad\text{ and }\quad \Sigma_{\text{i}} = \{\alpha \in \Sigma \mid
\frac{1}{2}\, \alpha \notin \Sigma\}\, .
\end{equation}
Both $\Sigma_\ast$ and $\Sigma_i$ are reduced root systems. Set $\Sigma_\ast^+ = \Sigma_\ast \cap \Sigma^+$. Note that $\bY$ is
irreducible if and only if $\Sigma_i$ is irreducible. Let $\Pi = \{\beta_j\}_{j = 1}^n$ be the
fundamental system of simple roots in $\Sigma^+_\ast$. 
For any $\lambda \in \g{a}_{\C}^\ast$ and $\alpha \in  \g{a}^\ast$ with $\alpha \ne 0$, define
$$\lambda_\alpha := \frac{\langle \lambda,\; \alpha \rangle}{\langle \alpha,\; \alpha \rangle}\, .$$
We will use similar notation for $i\fh^*$ without comments.
Note that $2\, \lambda_\alpha = \lambda_{\alpha / 2}$.
Define  $\omega_j\in\fa^\ast$, $j=1,\ldots ,n$, by
\begin{equation}
\label{eq6}
(\omega_{j})_{\beta_i} = \delta_{i, j},\quad 1 \leq i,\, j \leq n.
\end{equation}
The weights $\omega_j$ are the class $1$ fundamental weights for $(\g{u},\, \g{k})$ and $(\fg,\fk)$.
We let
\begin{eqnarray}
\label{eq1_6}
\Lambda^+_0=\{\lambda\in \fa^*\mid (\forall \alpha \in \Sigma^+)\,\, \lambda_\alpha\in\Z^+\}
= \sum_{j=1}^n\Z^+\omega_j\, .
\end{eqnarray}

For $\alpha \in \Sigma$ define the reflection $r_\alpha : \fa^*\to \fa^*$ by 
$$r_\alpha\, (\lambda) := \lambda - 2\, \lambda_\alpha\, \alpha,\; \forall\, \lambda \in \, \g{a}^\ast\, .$$ 
The group $W=W(\Sigma )$ generated by $r_\alpha$, $\alpha \in\Sigma$, is finite and
called   the Weyl group associated to $\Sigma$. 
Note that $W=W (\Sigma_\ast)=W (\Sigma_{\text{i}})$ with the obvious notation. 
The $W$-action extends to $\g{a}$ by duality, and then to $\g{a}_\C$ and
$\g{a}_{\C}^\ast$ by $\C$-linearity, and to $A$ and $A_\C$ by $w\cdot \exp (H)
= \exp (w(H))$. This action can be written as $w\cdot b=kbk^{-1}$ where $k\in N_K(\fa)$
is such that $\Ad (k)|_{\fa}=w$. 
The group $W$ then acts  on functions $f$ on any of these spaces by $(w\cdot f) (x) :=
 f (w^{-1}\cdot x)$, $w \in W$.  We recall that  $W\cdot \g{a}^+ = 
\g{a}^{\mathrm{reg}}$.  

We now describe the root structures for the special case of irreducible hermitian symmetric spaces
in more details.  Fix an orthogonal base $\{\eps_1,\ldots ,\eps_n\}$ for $\fa^*=i\fb^*$. 
Then we have the following description of the root system $\Sigma$, see
Moore \cite[Theorem 5.2]{mo},  \cite[p.528, 532]{h1}: 

\begin{theorem} 
\label{thm5} 
There are two possibilities for the root system $\Sigma^+$:
\begin{eqnarray*}
\text{Case I:} & & \Sigma^+ = \{\eps_j \pm \eps_i\, (1 \leq i < j \leq n),\; 2\, \eps_j\, (1 \leq
j \leq n)\}\\
\text{Case II:} & & \Sigma^+ = \{\eps_j\, (1 \leq j \leq n),\, \eps_j \pm \eps_i\, (1 \leq i < j \leq n),\, 
2\, \eps_j\, (1 \leq j \leq n) \}.
\end{eqnarray*}
\end{theorem}

\begin{remark} \rm
In Case II, $\Sigma$ is of type $BC_n$. Since $\bY$ is irreducible, $\Sigma = \mathcal{O}_s\, \dotcup\, \mathcal{O}_m\, \dotcup\, \mathcal{O}_l$ 
is a disjoint union of three $W$-orbits in $\Sigma$ corresponding to short, medium, and long roots, respectively.
Let $\mathcal{O}_s^+ = \mathcal{O}_s \cap \Sigma^+$, $\mathcal{O}_m^+ = \mathcal{O}_m \cap \Sigma^+$, and 
$\mathcal{O}_l^+ = \mathcal{O}_l \cap \Sigma^+$. Thus, 
\[\mathcal{O}_s^+ = \{\eps_j\, (1 \leq j \leq n)\},\quad \mathcal{O}_m^+ = \{\eps_j \pm \eps_i\, (1 \leq i < j \leq n)\}, \quad 
\mathcal{O}_l^+ = \{2\, \eps_j\, (1 \leq j \leq n)\}.\]
Adopt the notation 
$$m = (m_\alpha) = (m_s,\; m_m,\; m_l)$$
for root multiplicities of short, medium, and long roots, respectively, where
\[m_s = m_{\eps_i},\quad m_m = m_{\eps_j \pm \eps_i}\; (i \ne j), \quad m_l = m_{2\, \eps_i}.\]
In Case I, $\Sigma$ is actually of type $C_n$. We consider it as being of type $BC_n$ with $m_s = 0$. 
In this way, the root system $\Sigma$ is of type $BC_n$ in both cases.  
\end{remark}
 
The individual cases from Table (\ref{t1}) are listed in the following table:  

{\small \begin{equation}\label{t2}
\centering
\begin{tabular}[c]{|c||c|c|c|c|}
\hline
\multicolumn{5}{|c|}{ } \\
\multicolumn{5}{|c|}{Hermitian symmetric spaces $U / K$: Root structures and multiplicities}\\
\multicolumn{5}{|c|}{ } \\
\hline \hline
& $U$ & $K$ & $\Sigma$ & $(m_s,\, m_m,\, m_l)$\\
\hline
$1$ & $\mathrm{SU}\, (p + q)$ & $\mathrm{S}\, (\mathrm{U}_p \times \mathrm{U}_q)$ & 
$\begin{matrix}
\text{case I} & p = q\\
\text{case II} & q > p
\end{matrix}$ & 
$\begin{matrix}
(0,\, 2,\, 1)\\
(2\, (q-p),\, 2,\, 1)
\end{matrix}$ \\
\hline
$2$ & $\SO\, (2+q)$ & $\SO\, (2) \times \SO\, (q)$ & case I & $(0,\, q - 2,\, 1)$\\
\hline
$3$ & $\SO\, (2\, j)$ &  $\mathrm{U}\, (j)$ & 
$\begin{matrix}
\text{case I} & \text{$j$ is even}\\
\text{case II} & \text{$j$ is odd}
\end{matrix}$ & $\begin{matrix}
(0,\, 4,\, 1)\\
(4,\, 4,\, 1)
\end{matrix}$\\
\hline
$4$ & $\mathrm{Sp}\, (j)$ & $\mathrm{U}\, (j)$ & case I & $(0,\, 1,\, 1)$\\
\hline
$5$ & $\g{e}_{6\, (- 78)}$ & $\g{s o}\, (10) + \R$ & case II & $(8,\, 6,\, 1)$\\
\hline
$6$ & $\g{e}_{7\, (- 133)}$ & $\g{e}_6 + \R$ & case I & $(0,\, 8,\, 1)$\\
\hline
\end{tabular}
 \end{equation} }

\subsection{Basic structure theory}
Let $\g{n} = \oplus_{\alpha \in \Sigma^+}\, \g{g}_\alpha$. Then $\fn$ is a nilpotent Lie algebra and
 $\fg=\fk\oplus \fa\oplus \fn$.
Let $N = \exp\, \g{n}$ be the
analytic subgroup of $G$ with Lie algebra $\g{n}$. Then $N$ is nilpotent, simply
connected and closed.  The group $G$ is analytically diffeomorphic to $  K \times A \times N$ via the
multiplication map
$K \times A \times N \ni (k, a, n) \mapsto k a n \in G$. The inverse is denoted by
$x\mapsto (\kappa (x),a(x),n(x))$. We write $H(x)=\log (a(x))$. This is the Iwasawa decomposition
of $G$.  Furthermore $K_\C A_\C N_\C$ is open and dense in $G_\C$.

For the compact group $U$ we have the Cartan decomposition $U = K B K$. We have then the corresponding integral 
formulas for the Lie group $U$.

\begin{lemma}\label{le:int1} Let $\delta (\exp H)=\prod_{\alpha \in\Sigma^+}\left|\sin\left(\frac{1}{i}\alpha (H)\right)\right|^{m_\alpha}$
for $H \in \fb$. Then there exists a constant $c>0$ such that for all $f\in L^1(\bY)$
\[\int_{\bY} f(y)\, dy=c\int_K\, \int_{B} f(ka \cdot x_o)\delta (a)\, da\, d k\, .\] 
\end{lemma}

\section{Spherical Fourier Analysis on $\bY$}
\label{Fourier}
\noindent
In this section we recall the classification of $\chi_l$-spherical representations of $U$ (and $G$) due
to H. Schlichtkrull \cite{sc}. We then recall the Plancherel formula for $L^2(\bY;\cL_l)$ where $\cL_l=\cL_{\chi_l}$.
Next, we discuss the $\chi_l$-spherical functions and the decomposition of $L^2 (\bY;\cL_l)$.  

\subsection{The $\chi_l$-spherical representations}  
\label{s:sphRep}

Let $\Lambda^+ (U) \subset i\, \g{h}^\ast$ be the semi-lattice of highest weights of irreducible representations of $U$. 
As we are assuming that $U$ is simply connected we have
$$\Lambda^+ (U) = \left\{\lambda \in \g{h}_{\C}^\ast \mid (\forall \alpha \in \Delta^+)\,\,  2 \lambda_\alpha \in \Z^+\right\}\, .$$ 

For $\lambda \in \Lambda^+ (U)$ choose an irreducible unitary representation
$(\pi_\lambda,\, V_\lambda)$ of $U$. For $l\in \Z$ let
\begin{equation}\label{eq:chi}
V_\lambda^{\chi_l}=V_\lambda^l=\{v\in V_\lambda \mid (\forall k\in K)\,\, \pi_\lambda (k)v=
\chi_l (k)v\}\, .
\end{equation}
Note that $V_\lambda^0=V_\lambda^K=\{v\in V_\lambda \mid (\forall k\in K)\,\, \pi_\lambda (k)v=v\}$
is the space of $K$-fixed vectors. The representation $(\pi_\lambda,V_\lambda)$ is said to be
$\chi_l$-spherical if $V_\lambda^l\not=\{0\}$ and spherical if it is $\chi_0$-spherical. If $(\pi_\lambda, V_\lambda)$
is $\chi_l$-spherical then  $\dim V_\lambda^l =1$. 
Denote by $\Lambda_l^+(U)$ the set of highest weights of $\chi_l$-spherical representations of $U$. Let  
\begin{equation}\label{eq:Lplus}
\Lambda_l^+ := \{\mu \in \g{a}^\ast\; \mid\; \mu = \lambda |_\g{a},\; \lambda \in \Lambda_l^+\, (U)\}\, .
\end{equation}
 
According to \cite[p.535 and p.538]{h2} we have that if $\lambda\in\gL^+_0(U)$ then
$\lambda|_{\fh\cap \fk}=0$ and the set
$\Lambda_0^+$ is exactly the set introduced in (\ref{eq1_6}). 

According to  \cite{sc}, we can decompose $\g{h} \cap \g{k}$ as
$$\g{h} \cap \g{k} = (\g{h}    \cap \g{k}_1)  \oplus  \R  X$$
where $X$ is defined as in \cite[p.285, $(4.4)$]{sc} so that
\begin{enumerate}
\item $e^{t\, X} \in K_1$ if and only if $t \in 2 \pi\, i\, \Z$,
\item $Z - X \in \g{k}_1$ where $Z$ is the same as in Proposition \ref{pro0} (see Lemma $4.3$ in \cite{sc}).
\end{enumerate}
Note that $X=0$ in Case I. 
For $\lambda\in i\,\fh^*$ we write accordingly $\lambda = (\mu,\lambda_1,\mu_0)$ where 
$\mu=\lambda|_{\fa}$, $\lambda_1 = \lambda|_{\fk_1\cap \fh}$ and $\mu_0=\lambda (iX)$. If $X=0$ then
we write $\mu_0=0$. When $\chi_l$ is fixed for some $l \in \mathbb{Z}$, $\mu_0$ is then fixed. 
Note that $\mu_0$ is independent of $\lambda$. If $\lambda \in \Lambda_l^+\, (U)$ then $\lambda\, \big|_{\g{h} \cap \g{k}_1} = 0$.
Hence, $\lambda$ is uniquely determined by its restriction $\mu$. 
Thus there is a bijective correspondence $\Lambda_l^+\, (U) \cong \Lambda_l^+$ via $\lambda = (\mu, 0, \mu_0) \mapsto \mu$.
For convenience, we sometimes write $(\mu,\, 0,\, \mu_0) = \mu + \mu_0$.
We identify $\fa_\C^*$ with $\C^n$ by $\mu =(\mu_1,\ldots ,\mu_n)$ with
$\mu_j=\mu_{\eps_j}$.  Recall that $\fa^\ast =i\fb^\ast$.

\begin{theorem}[H. Schlichtkrull]
Let $U$ be a compact simply connected semisimple Lie group and $K$ the fixed point group of an involution of $U$. Let
$l\in \Z$. The set $\Lambda_l^+$ of highest restricted weights of irreducible $\chi_l$-spherical
representations of $U$ is given by 
\begin{equation}
\label{eq11}
\Lambda_l^+ = \left\{\mu \in \g{a}^\ast\; \left|\; 
\begin{array}{l}
\mu_j - \mu_i \in 2\, \Z^+\, (1 \leq i < j \leq n)\\
\mu_0 = 0\; \text{(Case I)};\; \mu_0 = l\; \text{(Case II)}\\
\mu_1 \in |l| + 2\, \Z^+
\end{array}\right.\right\}.
\end{equation}
\end{theorem}
\begin{proof} See Proposition 7.1 and Theorem 7.2 in \cite{sc}.
\end{proof}

\begin{remark}{\rm It follows from the description of $\Lambda_l^+$ that in Case I for some $l\in \Z$, a $\chi_l$-spherical representation 
$\pi_\mu$ is also spherical if $l$ is even. In fact it is shown in \cite[Thm. 7.2]{sc} that if $l$ is even $\pi_\mu$ 
must also contain the character $\chi_0$. }
\end{remark}

A simpler description of the set $\gL^+_l$ is given by the following proposition. For that let
\begin{equation}
\label{eqrhos}
\rho_s= \frac{1}{2}\sum_{\alpha\in\cO_s^+}\alpha = \frac{1}{2}\, \sum_{j=1}^n\, \eps_j.
\end{equation}

\begin{proposition}
\label{pro2}
For $l \in \Z$ we have $\Lambda_l^+ = \Lambda_0^+ + 2 |l|\, \rho_s $.
\end{proposition} 

\begin{proof}
Let $\mu \in \Lambda_l^+$.  We want to show that $(\mu - 2|l|\rho_s)_\alpha =\mu_\alpha-2|l|\, (\rho_s)_\alpha
\in \Z^+$ for all $\alpha \in \Sigma^+$. Let $r=|l|$. We have
\begin{equation}
\label{eq1pr}
\mu_{\eps_j}-2r (\rho_s)_{\eps_j}=\mu_j-r \in 2\Z^+
\end{equation}
and
\begin{equation}
\label{eq3pr}
\mu_{\eps_j\pm \eps_i} - 2r (\rho_s)_{\eps_j \pm \eps_i}=\frac{1}{2} (\mu_j\pm \mu_i - (r \pm r))\in \Z^+
\end{equation}
according to (\ref{eq11}). Finally, again by (\ref{eq11}), we have
\begin{equation}
\label{eq2pr}
\mu_{2\eps_j}-2r\, (\rho_s)_{2\eps_j}=\frac{1}{2}\left(\mu_j - r\right)\in \Z^+\, .
\end{equation} 
Thus $\mu-2|l|\rho_s\in\gL^+_0$.

On the other hand, if $\mu^0\in \gL^+_0$ define $\mu=\mu^0+2|l|\rho_s$. Then (\ref{eq1pr}), (\ref{eq3pr}) and
(\ref{eq2pr}) together with (\ref{eq11}) show that $\mu \in\gL^+_l$. 
\end{proof}
 
Recall that the fundamental spherical weights $\omega_j$ are defined by (\ref{eq6}). 
Then $\gL_0^+=\Z^+\omega_1 \oplus \dotsb \oplus \Z^+\omega_n$. Hence
\begin{equation}\label{eq:gLiso}
 (\Z^+)^n\simeq \gL_l^+\, ,\quad (k_1,\ldots ,k_n)\mapsto k_1\omega_1+ \dotsb + k_n \omega_n +2|l|\rho_s\, .
\end{equation}

\subsection{The Fourier transform}
In this section we recall the basic facts about Fourier analysis on $L^2(\bY;\cL_l)$. 
For $\lambda =(\mu ,0,\mu_0)\in \gL^+_l (U)$. As $l$, and hence $\mu_0$ will be fixed most
of the time, the only variable is the first coordinate $\mu$ and sometimes $(\mu,\, l)$. 
We therefore simply write  $\mu$ instead of $\lambda$.

Let $L^2(\bY;\cL_l)$ be the space of $L^2$-sections of the line bundles $\cL_l$,
$L^2(U//K; \cL_l)$ the space of elements in $L^2(\bY;\cL_l)$ such that $f(k_1uk_2)=\chi_l (k_1k_2)^{-1}f(u)$
for all $k_1,k_2\in K$ and $u\in U$. Finally $C^\infty (U//K; \cL_l)$ is the space of smooth elements in 
$L^2(U//K;\cL_l)$.

Let $d (\mu ) := \dim\, V_\mu$. Then $\mu\mapsto d (\mu )$ is a polynomial map. 
Fix $ e_{\mu,l} \in V_\mu^{l}$ of length one. We will mostly write $e_\mu$ for $e_{\mu,l}$ as $l$ will be fixed.
 
We normalize the invariant measure on all compact groups so that the total measure is one. Define
\begin{equation}\label{Pchi}
P_{\mu,l}(u):=\int_K \chi_l(k)^{-1}\pi_\mu (k)u\, dk\, .
\end{equation}
Then $P_{\mu,l}$ is the orthogonal projection: $V_\mu \to V_\mu^l$. In
particular, $P_{\mu,l}(u)=\ip{u}{e_\mu}\, e_\mu$. If
$f\in L^2(\bY;\cL_l)$ then $\pi_\mu (f)=\pi_\mu (f)P_{\mu,l}$,
where, as usually,
\[\pi_\mu (f)= \int_U f(u)\pi_\mu (u)\, du.\]
It is therefore natural to define the vector valued Fourier transform of $f$ to be
\begin{equation}\label{eq:Fourier}
\widehat{f}(\mu,l)=\widehat{f}(\mu)=\pi_\mu (f)e_\mu\, .
\end{equation}
Note that 
\[\Tr (\pi_\mu (f))=\ip{\pi_\mu (f)e_\mu}{e_\mu}=
\int_{U} f(u)\ip{\pi_\mu (u)e_\mu}{e_\mu}\, du =\ip{f}{\psi_{\mu,l}}.\]
The function
\begin{equation}\label{eq:psi}
\psi_{\mu,l}(u)=\ip{e_\mu}{\pi_\mu (u)e_\mu},\; u \in U
\end{equation}
is the $(\mu,l)$, or $\chi_l$, \textit{spherical function} on $U$ which we will discuss in more details in
the next section. 
Furthermore, if  $f\in L^2 (U//K; \cL_l)$ then $\pi_\mu (f)e_\mu$ is again a scalar multiple of $e_\mu$ and so
\[\pi_\mu (f)e_\mu =\ip{\pi_\mu (f)e_\mu}{e_\mu}\, e_\mu=
\ip{f}{\psi_{\mu,l}}\, e_\mu.\]
The $\chi_l$-spherical function is the unique element in $C^\infty (U// K;\cL_l)$ such that
$\psi_{\mu,l}(e)=1$. Furthermore $\{d(\mu)^{1/2}\, \psi_{\mu,l}\}$  is an orthogonal basis for
$L^2(U//K;\cL_l)$. Note that
\[\overline{\psi_{\mu,l}(u)}=\ip{\pi_\mu (u)e_\mu}{e_\mu}=\ip{e_\mu}{\pi_\mu (u^{-1})e_\mu}=\psi_{\mu,l}(u^{-1})\]
Let
\begin{equation}\label{eq:ell}
\ell^2_{d} (\gL^+_l) =\{(a(\mu))_{\mu\in\gL^+_l}\mid \sum_{\mu \in \gL_l^+} d(\mu )|a(\mu)|^2<\infty\}\, .
\end{equation}
Then $\ell^2_d(\gL^+_l)$ is an Hilbert space with inner product $(a, b)=\sum d(\mu) a(\mu)\overline{b (\mu)}$.
The \textit{$\chi_l$-spherical Fourier transform} $\cS_l: L^2(U//K,\cL_l)\to \ell^2_d(\gL^+_l)$ defined by
\begin{equation}\label{eq:SpFt}
\cS_l\, (f)\, (\mu) = \int_U\, f\, (u)\, \psi_{\mu,\, l}\, (u^{-1})\, d\, u = \ip{f}{\psi_{\mu,\, l}}
\end{equation}
is an unitary isomorphism. We collect the main facts in the following theorem:

\begin{theorem}[The Plancherel Theorem]\label{pro13}
Assume that $\bY$ is a simply connected. For $\mu \in\Lambda_l^+(U)$ and
$v\in V_{\mu}$
let $f_{\mu,v}(x)=\ip{v}{\pi_{\mu}\, (x)\, e_\mu}$, $v\in V_\mu$ and  $L_\mu^2(\bY;\cL_l)=\{f_{\mu ,v}\mid
v\in V_{\mu}\}$. Then the following holds true:

(1) If $f\in L^2(\bY;\cL_l)$ then
\begin{eqnarray*}
\|f\|^2 & = &\sum_{\mu \in \gL^+_l} d(\mu)\|\wh{f}\|_{V_\mu}^2\\
f(x) & = & \sum_{\mu\in\gL^+_l} d (\mu ) \ip{\wh{f}(\mu)}{\pi_\mu (x)e_\mu}
\end{eqnarray*}
where the convergence is in the $L^2$-norm topology. The convergence is uniform if $f$ is smooth.

(2) $L^2(\bY;\cL_l)\simeq \bigoplus_{\mu\in \gL^+_l} L^2_\mu(\bY ; \cL_l)$.  

(3) If $f\in L^2(U//K;\cL_l)$ then
\begin{eqnarray*}
\|f\|^2 & = &\sum_{\mu\in\gL^+_l} d(\mu )|\cS_l\, (f)\,(\mu)|^2\\
f & = & \sum_{\mu\in\gL^+_l}\, d(\mu ) \cS_l\, (f)\, (\mu )\psi_{\mu,l}
\end{eqnarray*}
where the sum is understood in the $L^2$-norm sense and uniformly if $f$ is smooth.

(4) $L^2(U//K;\cL_l)\simeq \ell^2_d(\gL^+_l)$.
\end{theorem}

\section{The $\chi_l$-Spherical Functions}\label{s:spheri}
\noindent
The $\chi_l$-spherical functions were already introduced in the last section. We now
discuss them in more details 
and present the results needed for the proof of the 
Paley-Wiener Theorem in Section \ref{s:PW}. The standard
references for the material in this section are \cite{hs}. Also see \cite{sh}. Our assumptions are the
same as in the last section. In particular $\bY=U/K$ is an irreducible Hermitian symmetric 
space with $U$ simple and simply connected.

\subsection{The $\chi_l$-spherical functions on $G$} Let us start by recalling the definition of
a $\chi_l$-spherical function.

\begin{definition}
Let $H$ be a locally compact Hausdorff group and $L \subset H$ a compact subgroup.  Let $\chi_l: L\to \T$ be
a continuous homomorphism. A continuous function $\varphi: H
\to \C$ is an \textit{(elementary) spherical function of type $\chi_l$} if $\psi$ is not identically $0$
and
\begin{equation}
\label{eq13}
\int_L  \varphi (a k  b)\, \chi_l (k)\, d k = \varphi (a) \varphi (b)\, ,\quad \text{ for all }\quad  a,b\in H\, .
\end{equation}
\end{definition}
We will mostly say that $\varphi$ is a $\chi_l$-spherical function or a spherical function of type $\chi_l$.

\begin{lemma} Let $\varphi $ be a spherical function of type $\chi_l$. Then
$$\varphi  (k_1 h  k_2) = \chi_l  (k_1 k_2)^{-1}\, \varphi (h )\, , \quad \text{ for all } \quad  h \in H \text{ and } k_1, k_2 \in L\, .$$
Furthermore, $\varphi (e)=1$. 
\end{lemma}

\begin{proof} Let $b\in H$ be so that $\varphi (b)\not= 0$. Let $a\in H$ and $m\in L$. Then 
\[\varphi (am) = \frac{1}{\varphi (b)}\int_L \varphi (amkb)\chi_l (k)\, dk 
 =  \frac{1}{\varphi (b)}\int_L \varphi (akb)\chi_l (m^{-1}k)\, dk = \chi_l (m)^{-1}\varphi (a)\, .\]
One can show that $\varphi (ma)=\chi_l (m)^{-1}\varphi (a)$ in the same way by applying (\ref{eq13}) to 
\[\frac{1}{\varphi (b)}\int_L \varphi (bka)\chi_l (k\, m^{-1})\, dk.\]
That $\varphi (e)=1$ follows from (\ref{eq13}) by taking $b=e$.
\end{proof}
 
As $G_\C=U_\C$ is simply connected it follows that
$\pi_\mu$ extends to an irreducible holomorphic representation of $G_\C$ which we also denote by
$\pi_\mu$. Thus  $\chi_l$ extends to a homomorphism of $K_\C$ also denoted by
$\chi_l$.

\begin{lemma}
Let $\mu \in \Lambda_l^+$. Then $\psi_{\mu,l}$ is a spherical function of $\chi_l$ type.
It extends to a holomorphic function on $U_\C$. The extension is given by
$$\check{\psi}_{\mu,\, l}\, (g) = \ip{\pi_\mu (g^{-1}) e_\mu}{ e_\mu}\, , \qquad g \in U_\C\, .$$
Furthermore, the holomorphic extension satisfies
$\check{\psi}_{\mu,l}(k_1gk_2)=\chi_l (k_1k_2)^{-1}\check{\psi}_{\mu, l} (g)$ for all
$k_1,k_2\in K_\C$ and $g\in U_\C$.
\end{lemma}
\begin{proof} This is standard, but let us show that $\psi_{\mu,l}$ satisfies (\ref{eq13}). For that we note that
\[\int_K \chi_l (k)^{-1}\pi_{\mu}(k)\pi_{\mu }(b)e_\mu\, d\, k =\ip{\pi_\mu (b)e_\mu}{e_\mu}\, e_\mu,\quad b \in U. \]
Hence, for $a, b \in U$,
\begin{eqnarray*}
\int_{K}\psi_{\mu,l}(akb)\chi_l\, (k)\, dk&=&\int_K \ip{\pi_{\mu}(a^{-1})e_\mu}{\chi_l\, (k)^{-1}\pi_\mu (k)\pi_\mu (b)
e_\mu}\, dk\\
&=&\ip{e_\mu}{\pi_\mu (a)e_\mu}\overline{\ip{\pi_\mu (b)e_\mu}{e_\mu}}\\
&=& \psi_{\mu,\, l} (a)\, \psi_{\mu,\, l} (b)\, .
\end{eqnarray*}
\end{proof}

For $\lambda \in \g{a}_\C^\ast$ define $\varphi_{\lambda,\, l}: G \to \C$ by 
\begin{equation}
\label{eq5}
\varphi_{\lambda,\, l} (g) = \int_K a (g^{-1} k)^{\lambda - \rho} \chi_l (\kappa (g^{-1} k) k^{-1})\, d k.
\end{equation}

\begin{remark}
\label{rm:varphi}{\rm
Notice that the formula (\ref{eq5}) differs from the one  in \cite[p.82, (5.4.1)]{hs} by
an inverse sign. The definition (\ref{eq5}) for $\varphi_{\lambda, l}$ is equivalent to
$$\varphi_{\lambda,\, l}(g) = \int_K  a (g k)^{- \lambda - \rho} \chi_l (\kappa  (g k)^{-1} k)\, d k.$$
When $l = 0$, $\varphi_\lambda\, (g) = \varphi_{\lambda,\, 0}\, (g)$ is  the  Harish-Chandra spherical
function on $G$. }
\end{remark}

For the following theorem see \cite[p.82, Proposition 5.4.1]{hs}, \cite[Proposition 3.3, Corollary 3.7]{sh}, and Remark \ref{rm:varphi}:

\begin{theorem}\label{th:symm} The function $\varphi_{\lambda,\, l}$ is a spherical function of type $\chi_l$ on $G$.
If $\psi$ is a spherical function of type $\chi_l$ then there exists $\lambda\in \fa_\C^*$ such 
that $\psi=\varphi_{\lambda,\, l}$. Furthermore the following holds true:
\begin{enumerate}
\item  $\varphi_{\lambda,\, l}\, (g)$ is real analytic in $g \in G$, and holomorphic in
$\lambda  \in \g{b}_{\C}^\ast$.
\item $\varphi_{\lambda,\, l} = \varphi_{\mu, l}$ if and only if there is a $w \in W$ such that $\lambda = w \mu$.
\item $\varphi_{\lambda,\, l}\, (w\, a) = \varphi_{\lambda,\, l}\, (a)$, $\forall\, w \in W$.
\item $\varphi_{\lambda, l} (a) = \varphi_{\lambda, - l} (a) = \varphi_{-\lambda, - l} (a) = 
\varphi_{-\lambda,l}(a)$, $\forall\, a \in A$. 
\item $\varphi_{\lambda, l} (g) = \varphi_{-\lambda, -l}(g^{-1}) =
 \varphi_{\lambda, -l} (g^{-1}) = \varphi_{- \lambda, l} (g)$, $\forall\, g \in G$.
\end{enumerate}
\end{theorem}
  
The following lemma gives one of the main steps to analytically continue the $\chi_l$-spherical Fourier transform.

\begin{lemma}
\label{lem1}
Let $\mu \in \Lambda_l^+$. Then $\varphi_{\mu+\rho,l}$ extends to
a holomorphic function on $G_\C$, denoted again by $\varphi_{\mu+\rho,l}$, and $\check{\psi}_{\mu,l}=\varphi_{\mu+\rho,l}$. 
\end{lemma}

\begin{proof} As $G$ is totally real in $G_\C$ it is enough to show
that $\check{\psi}_{\mu,l}|_G = \varphi_{\mu+\rho, l}$. Let $u_\mu\in V_{\mu,l}$ be a nonzero highest weight vector. Then $V_{\mu,l}$
is generated by $\pi_\mu (K_\C)\, u_\mu$.  As $K_\C A_\C N_\C$ is dense it follows that $\ip{u_\mu}{e_\mu}\not= 0$.
Choose $u_\mu$ so that $\ip{u_\mu}{e_\mu}=1$. Then $P_{\mu,l}u_\mu =e_\mu$. Thus for $g\in G$:
\begin{eqnarray*}
\check{\psi}_{\mu,l} (g)&=& \ip{\pi_\mu (g^{-1})e_\mu}{e_\mu}\\
&=& \int_K \ip{\chi_l(k)^{-1}\pi_\mu (g^{-1}k)u_\mu}{e_\mu}\, dk\\
&=&\int_K  a(g^{-1}k)^{\mu} \chi_l(\kappa\, (g^{-1}\, k)\, k^{-1}) \, dk\\
&=& \varphi_{\mu +\rho,\, l}(g)\, .
\end{eqnarray*}
\end{proof}
  
\subsection{Holomorphic extension and estimates for $\varphi_{\lambda,l}$}\label{s:holExt}
We refer to Chapter 5 in \cite{hs} and Section 2 in \cite{sh} for detailed discussion about invariant differential operators on $\cL_l$. Here we
just recall what we need. Let $\D_l (\bX)
\simeq \D_l (\bY)$ be the algebra of invariant differential operators
$D : C^\infty (\bX;\cL_l)\to C^\infty (\bX;\cL_l)$.  Let $U(\fg )^K$ be the $\Ad (K)$-invariant elements in
the universal enveloping algebra of $\fg_\C=\fu_\C$. Then there exists a surjective
map $u\mapsto D_u$ of $U(\fg )^K$ onto $\D_l (\bX)$. We denote by $\gamma_l : U(\fg )^K\to S(\fa )^W$ the
Harish-Chandra homomorphism. Here 
$S(\fa )^W$ is the commutative algebra of $W$-invariant polynomials on $\fa_\C^\ast$. Then
$\gamma_l$ induces a algebra isomorphism $\D_l (\bX)\simeq S(\fa )^W$, see \cite[Thm. 5.1.10]{hs}. Define
a homomorphism $\zeta_{\lambda ,l}: \D_l(\bX)\to \C$ by $\zeta_{\lambda ,l}(D)=\gamma_{l}(D)(\lambda )$.
We also write $\zeta_{l}(D;\lambda)$ for $\zeta_{\lambda, l}(D)$. 
We then have: 
\begin{lemma}[Theorem 3.2, \cite{sh}] Let $D\in \D_l (\bX)$ and $\lambda \in \fa_\C^*$. Then
\begin{equation}\label{eq:sDiff}
D\varphi_{\lambda, l}=\zeta_{l}(D;\lambda)\varphi_{\lambda,\, l}\, .
\end{equation}
If $\varphi\in C^\infty(U//K;\cL_l)$ is a solution of the system of differential equations (\ref{eq:sDiff}) and
$\varphi (e)=1$, then $\varphi =\varphi_{\lambda,\, l}$.
\end{lemma}

For $l\in \Z$ define 
\begin{equation}
\label{eq:eta}
\eta_l^+ := \prod_{\alpha \in \mathcal{O}_s^+} \left(\frac{e^\alpha + e^{- \alpha}}{2}\right)^{2|l|},\qquad 
\eta_l^- := \prod_{\alpha \in \mathcal{O}_s^+} \left(\frac{e^\alpha + e^{- \alpha}}{2}\right)^{-2|l|}\, .
\end{equation}
We will write $\eta_l$ for $\eta_l^+$.

Recall that a multiplicity function $m$ is a $W$-invariant functions $m: \Sigma \to \C$. In our case it can only take three
values $m=(m_s,m_m,m_l)$ \footnote{Our multiplicity notation is different from the one used by Heckman and Opdam.
The root system $R$ they use is related to our
$2\Sigma$, and the multiplicity function $k$ in Heckman and Opdam's work is related to our $m$ by 
$k_{2\alpha}=\frac{1}{2}m_\alpha$.}. It is possible that one or more of those
number is zero. For $l \in \Z$ define multiplicity functions 

\begin{eqnarray}
m_+(l)&=& (m_s-2|l|,m_m,m_l+2|l|)\label{mplus}\\
m_-(l)&=& (m_s+ 2|l|,m_m,m_l-2|l|)\label{mminus}\, .
\end{eqnarray}
We will also write $m (l)$ for $m_+(l)$.

Note that the radial part of the Laplace-Beltrami operator on $\bX$ (acting on $\chi_l$-covariant functions) is 
exactly the operator
$$L\, (l) = L\, (m (l)) := \sum_{j = 1}^n\, \partial_j^2 + \sum_{\alpha \in \Sigma^+}\, m (l)_\alpha\, \frac{1 + e^{- 2
\alpha}}{1 - e^{-2 \alpha}}\, \partial_\alpha$$
associated with the root system $\Sigma$ and the multiplicity $m (l)$. This operator is actually defined on 
$A_{\C}^{\text{reg}}=\exp (\fa^{\text{reg}}_\C)$. We write $\rho (l) = \rho\, (m\, (l))$. It was shown that 
$$\zeta_l\, (L\, (l);\, \lambda) = (\lambda,\, \lambda) - ( \rho (l),\, \rho (l) ).$$

\begin{theorem}\label{th:estim} Let 
\begin{equation}
\label{eq:Omega}
\Omega=\{X\in \fb\; \mid\; (\forall \alpha \in \Sigma)\,\, |\alpha (X)|<\pi\}.
\end{equation}
The function $A\times \fa_\C^*\to \C$,
$(a,\lambda )\mapsto \varphi_{\lambda,l}(a)$, extends to a holomorphic function
$(b,\lambda )\mapsto \varphi_{\lambda,l}(b)$ on $A\exp (\Omega )\times \fa_\C^*$.
The extension satisfies the symmetry conditions in Theorem \ref{th:symm}.
Furthermore there exists a constant $C>0$ such that for $X\in \Omega$ we have
\[|\varphi_{\lambda ,l}(\exp X)|\le C e^{\|X\|\|\re (\lambda )\|}\, .\]
\end{theorem} 

\begin{proof}
This is done in the Appendix. The estimate for $\varphi_{\lambda ,l}$ follows from Remark \ref{eta:estim} and Proposition \ref{hypfun:estim}.
\end{proof}

Let $\eps > 0$ and $\Omega_\eps=\{X\in \fb\mid (\forall \alpha \in \Sigma)\,\, |\alpha (X)|\leq \pi - \eps\}$. 
Then $\Omega = \cup_{\eps > 0}\, \Omega_\eps$.
Let $Y_j\in \fa$ be such that $\eps_i(Y_j)=\delta_{i,j}$. For $Y\in \fa_\C$ write
\[Y=\sum_{j=1}^n y_j Y_j=\sum_{j=1}^n \eps_j(Y)Y_j\, .\]
We have
\[\fa+\Omega_{\eps}=\{Y\in \fa_\C\; \mid\;  | \im\; y_j|\leq \frac{1}{2}(\pi -\eps)\}\, .\]

\section{The Paley-Wiener Theorem}\label{s:PW}
\noindent
Let $\|\; \,\|$ be the norm on $\g{u}$ with respect to the inner product $\ip{\,\cdot\,}{\,\cdot\,}$ defined earlier using
the Cartan-Killing form.
For $r > 0$, let $B_r\, (0) = \{X \in \g{q}:\; \|X\| < r\}$ be the open ball in $\g{q}$ centered at
$0$ with radius $r$.  Let $B_r(x_o)=\Exp (B_r(0))$. We will fix $R>0$ such that $R$ is smaller than the
injectivity radius and $B_R(0)\cap \fb\subset \Omega$. In particular $\Exp : B_r(0)\to B_r(x_o)$ is a diffeomorphism
and $\varphi_{\lambda ,l} $ is well defined on the closure of $B_r(x_o)$ for all $0<r<R$. We therefore let 
$\overline B_r(0)$ be the closed ball in $\fq$ with
radius $r$ and $\overline B_r(x_o)=\Exp (\overline B_r(0))$ the closure of $B_r(x_o)$. Finally 
we let $C^\infty_r(U//K; \cL_l)$ be the space of functions in $C^\infty (U// K;\cL_l)$ with support in
$\overline B_r(x_o)$. As $R$ is smaller than the injectivity radius it follows that for $0<r<R$ we have 
\begin{equation}\label{eq:res}
C_r^\infty\, (B)^W \cong C_r^\infty\, (B \cdot o)^W
\quad\text{ and }\quad C_r^\infty\, (U / / K;\, \cL_l) \xlongrightarrow{\cong} \eta_l \cdot C_r^\infty\, (B)^W
\, .
\end{equation}

For $r > 0$ denote by $\PW_r (\g{b}_{\C}^\ast)$ the space of holomorphic functions $\varphi$ on
$\g{b}_\C^\ast$ of exponential type $r$. Thus a holomorphic function $F$ on
$\fb_\C^\ast$ is in $\PW_r (\fb_\C^\ast)$ if and only if
  for every $k \in \mathbb{N}$, there is a constant $C_k$ such that
$$|\varphi\, (\lambda)| \leq C_k\, (1 + \|\lambda\|)^{- k}\, e^{r\, \|\re\, \lambda\|},\quad \forall\, \lambda \in
\g{b}_\C^\ast\, .$$

There are two natural actions of the Weyl group. The first one is the usual conjugation
of the variable, and the second is the $\rho$-shifted affine action 
$R(w)F(\lambda )=F(w^{-1}(\lambda +\rho )-\rho)$. Let
\[\PW_r(\fb_\C^\ast)^{R(W)}=\{F\in \PW_r(\fb_\C^\ast)\mid (\forall w\in W)\,\,
R(w)F=F\}\, .\] 
We note that this is the same Paley-Wiener space as in \cite{os}.  

Similarly one defines the space $\PW_r(\fb_\C^\ast)^W$ where we now use the standard
action of the Weyl group. We note that those spaces are isomorphic via the map
\[F\mapsto \Psi (F): \lambda \mapsto F(\lambda -\rho)\]
with inverse
\[G\mapsto \Psi^{-1}(G): \lambda \mapsto G(\lambda +\rho)\, .\]
To see that $\Psi (F)$ is $W$-invariant a simple calculation gives:
\[\Psi (F)(w\lambda )= F(w\lambda - \rho)
=F(w(\lambda -\rho + \rho)-\rho)=
F(\lambda -\rho )=\Psi (F)(\lambda )\, .\]
Similarly for $\Psi^{-1}(G)$.

We will use this isomorphism in the following to connect results from \cite{ow} on
restriction of Paley-Wiener spaces to subspaces without comment that one has sometimes to
use the above isomorphism.

\begin{theorem}[Paley-Wiener Theorem]
\label{thm1}
The (extended) $\chi_l$-spherical Fourier transform $\mS_l$ gives a linear bijection
\begin{equation}
\label{eq22}
\mS_l: C_r^\infty\, (U / / K;\; \cL_l)\; \stackrel{\cong}{\longrightarrow}\; \PW_r\, (\g{b}_{\C}^\ast)^{R(W)}
\end{equation}
for each $0 < r < R$. Precisely,
\begin{enumerate}
\item If $f \in C_r^\infty\, (U / / K;\, \cL_l)$, then $\mS_l\, (f): \Lambda_l^+ \to \C$ extends to a function in
$\PW_r\, (\g{b}_{\C}^\ast)^{R(W)}$;
\item Let $\varphi \in \PW_r\, (\g{b}_{\C}^\ast)^{R(W)}$. There exists a unique $f \in C_r^\infty\, (U//K;\, \cL_l)$
such that 
\[\mS_l\, (f)\, (\mu) = \varphi\, (\mu),\quad \forall\, \mu \in \Lambda_l^+;\]
\item The functions in $\PW_r\, (\g{b}_{\C}^\ast)^{R(W)}$ are uniquely determined by their values on
$\Lambda_l^+$.
\end{enumerate}
\end{theorem}

\begin{corollary}\label{co:4.3} Let $l,k\in\Z$ and $0<r<R$. Then 
\[\cS_k^{-1}\circ \cS_l :C_r^\infty (U// K; \cL_l) \cong C_r^\infty\, (U//K;\, \cL_k)\]
is a linear isomorphism. 
\end{corollary}

\begin{proof}
This follows from Theorem \ref{thm1} applied to $l$ and $k$.
\end{proof}

\begin{remark}{\rm 1) As remarked in \cite[Remark 4.3]{os} one can use different $R$ in (1), (2) and (3), 
and then  take the minimum of those constants for the map (\ref{eq22}) to be a bijection. 
\smallskip

2) In \cite{os} the authors  used for $\Omega$ the domain where
$|\alpha (X)|\leq \pi/2$. This is because \cite{os} used the Opdam estimates \cite{op} which were shown for this domain.
\smallskip

3) We note that  $C^\infty_r(\bY)^K=C^\infty_r(U// K; \cL_0)$ so this case is cowered in the corollary.}
\end{remark}

The hard part of Theorem \ref{thm1} is (2) so we start with (1) and (3) and leave (2)
for the next section. The proof follows closely \cite{os}. 

\begin{proof}(Part (1)) Let $\mu \in \gL^+_l$ and $f\in C_r^\infty (U//K; \cL_l)$. It is easy to see that the function 
$u\mapsto f(u)\overline{\psi_{\mu, l} (u)}$ is $K$-biinvariant. Using Lemma \ref{le:int1} and using that $-1\in W$, we get
\begin{eqnarray*}
\cS_l(f)(\mu)&=&\frac{1}{|W|} \int_B f(b)\delta (b)\overline{\psi_{\mu, l} (b)} \, db\\
 &=& 
\frac{1}{|W|} \int_B [f(b)\delta (b)]\psi_{\mu, l} (b )\, db\\
&=&
\frac{1}{|W|} \int_B [f(b)\delta (b)]\varphi_{\mu+\rho, l} (b )\, db\, .
\end{eqnarray*}

As $\supp (f|_{B\cdot x_o})\subseteq \exp (\Omega )\cdot x_o$ we can, using Theorem \ref{th:estim}
define the holomorphic extension of $\cS_l (f)$ to $\fa_{\C}^\ast$ by 
\[\lambda \mapsto \cS_l (f)(\lambda) =\frac{1}{|W|} \int_B [f(b)\delta (b)]\varphi_{\lambda+\rho,l} (b )\, db\, .\]
Then, by the $W$-invariance of $\varphi_{\lambda,\, l}$ it follows that
$\cS_l (f)(w(\lambda +\rho )-\rho)=\cS_l (f)(\lambda)$.

As $b \mapsto f(b)\delta (b)$ is in $C^\infty_r (B)^W$ it follows again from Theorem \ref{th:estim}
that 
\[|\cS_l (f)(\lambda) |\le C e^{r\|\re\, \lambda\|}\, .\]
The polynomial estimate follows by applying $D\in\cD_l (\bY)$ to $\varphi_{\lambda,l}$ and noticing that 
\[\zeta_l(D;\lambda )\cS_l(f) (\lambda )=\int_{\bY} f(y) D \varphi_{\lambda,l} (y^{-1})\, dy
= \int_\bY D^* f(y) \varphi_{\lambda,l} (y^{-1})\, dy
=\cS_l(D^* f) (\lambda)
\] 
 where $D^*$ is the adjoint of $D$. 
\end{proof}

\begin{proof}(Part (3)) This follows from a generalization of Carleson's theorem for
higher dimensions, see \cite[Lem. 7.1]{os}. Here we use the fundamental weights
$\omega_1,\ldots ,\omega_n$ and (\ref{eq:gLiso}) to view $\cS_l(f)$ as a function
on $\C^n$. Denote by $\|\lambda \|_0$ the standard norm on $\C^n$. Then there exists
$C>0$ such that $ \|\lambda\|\le C\|\lambda \|_0$. It follows that there exists $C_1>0$ such
that
\[ |\cS_lf(\sum \lambda_j \omega_j + 2\, |l|\, \rho_s)|\le  C_1e^{(rC) \| \lambda \|_0}\, .\] 
Hence, if $rC<\pi$ and $\cS_lf(\sum k_j \omega_j + 2\, |l|\, \rho_s)=0 $ for all $k_j\in\Z^+$ then Carleson's theorem implies that $\cS_lf =0$ and hence the extension is unique.
\end{proof}

\section{The Surjectivity}
\label{s:bij}

\noindent
In this section we prove part (2) of Theorem \ref{thm1}.
The proof is reduced to the Paley-Wiener theorem
for central functions on compact Lie group originally proved by F. Gonzalez in \cite{go}. 
The reduction depends on a surjectivity criteria for restriction of Paley-Wiener spaces.

\subsection{The Paley-Wiener theorem for central functions on $U$}\label{s:pwtgp}
This is a special case of the Paley-Wiener theorem for compact symmetric spaces as $U\simeq U\times 
U/K$ where $K$ is the diagonal in $U \times U$, i.e. $K=\{(u,u)\mid u\in U\} \simeq U$. The corresponding involution is
$\tau (a,b)=(b,a)$ and the action of $U\times U$ on $U$ is $(a,b)\cdot u=aub^{-1}$. In particular we have
\[C^\infty (U)^U =\{f\in C^\infty (U)\mid (\forall u,k\in U)\, \, f(kuk^{-1})=f(u)\}\, . \]
The spherical functions on $U=U\times U/K$ are the normalized trace functions
\[\xi_\mu (u)=\frac{1}{d(\mu)}\, \Tr (\pi_\mu (u^{-1}))\, .\]
The noncompact dual is $G/K=U_\C/K$. The role of $\fa$ is played by the Cartan
subalgebra $i\, \fh$, $W$ by $W_\fh$, the Weyl group associated to $\Delta$, and $\gL_0^+$ by
$\gL^+(U)$, the semi-lattice of all highest weights of irreducible representations of
$U$. Finally, the spherical functions on $U_\C/K$ are given by, see \cite[Thm. 5.7, Chapter IV]{h2}:
\[\varphi_\lambda (a) =\frac{\pi (\rho (\fh ))}{\pi (\lambda )}\, \
\frac{\sum_{w\in W} (\det w)a^{w\lambda}}{\sum_{w\in W} (\det w)a^{w\rho_\fh}}\, ,\quad
\text{ where } \quad \pi (\mu )=\prod_{\alpha\in \Delta^+}\ip{\alpha}{\mu}\, .\]

The result of \cite{go} was used by \cite{os} to prove the surjectivity part of the Paley-Wiener theorem for $K$-invariant
functions on $U/K$. The simple reformulation corresponding to results of \cite{os}, i.e., using
the normalized trace function $\xi_\mu$ instead of the character $\Tr \circ \pi_\mu$, was done
in \cite[Lem. 5.4]{ow}. The formulation of Gonzalez theorem, in the form we need it, is then:

\begin{theorem}\label{th:PWgroup} Let $\mathrm{PW}_r(\fh_\C^\ast)^{R(W_\fh)}$ be the space of holomorphic functions on
$\fh_\C$ of exponential growth $r$ and such that $F(w(\lambda +\rho_\fh)-\rho_\fh)=
F(\lambda)$ for all $\lambda \in \fh_\C^*$ and $w \in W_\fh$. Then there exists a $R>0$ such that
for all $0<r<R$ the spherical Fourier transform 
\[\widetilde{f}(\lambda ) =\int_U f(u) \varphi_{\lambda + \rho}\, (u^{-1})\, du\]
is a surjective linear map $C_r^\infty(U)^U\to \mathrm{PW}_r(\fh_\C^*)^{R(W_\fh)}$.
\end{theorem}

\subsection{The Surjectivity of the $\chi_l$-Spherical Fourier Transform}
\label{s:surj}
It is easy to see that if $F\in \mathrm{PW}_r(\fb^*_\C)^{R(W)}$ then the function
\[f(u)=\sum_{\mu \in \gL_l^+} d(\mu ) F(\mu)\psi_{\mu, l}(u)\]
is smooth and $\cS_l (f)=F$. The hard part is to see that $\supp (f)\subseteq
\overline{B}_r(x_o)$.
To use Theorem \ref{th:PWgroup} we define
\[Q_l (f)(u)=\int_K f(uk)\chi_l (k)\, dk=
\int_K \chi_l (k)f(ku)\, dk\, ,\quad f\in C_r^\infty (U)^U\, .\]

\begin{lemma} 
\label{lemma1}
Assume that $ f\in C_r^\infty\,(U)^U$, then $Q_l(f)\in C_r^\infty (U// K,\cL_l)$.
Furthermore, if $F\in
\mathrm{PW}_r(\fh^*_\C)^{R(W_\fh)}$ and
$f=\sum_{\mu\in \gL^+(U)} d(\mu )^2 F(\mu )\xi_\mu$, then
\begin{equation}\label{eq:Ql}
Q_l(f)(u)=\sum_{\mu\in \gL_l^+} d((\mu,0,\mu_0))F( (\mu,0,\mu_0)) \psi_{\mu, l}(u)\, .
\end{equation}
\end{lemma}

\begin{proof}
If $f=\sum_{\mu\in \gL^+(U)} d(\mu )^2 F(\mu )\xi_\mu$ then $f$ is 
the inverse Fourier transform of $F$. Because of the
rapid decay of $F$ it follows that
\[Q_l(f)=\sum_{\mu\in \gL^+(U)} d(\mu )^2 F (\mu) Q_l(\xi_\mu )\, .\]
Note that the square in $d(\mu )^2$ comes from the fact that the representation
that we are in fact using in $V_\mu \otimes V_\mu^*$ has dimension
$d(\mu )^2$.

Recall that the $P_{\mu,l}v=\int_K \chi_l(k)^{-1}\pi_\mu (k)v$ is
the orthogonal projection $V_\mu \to V_\mu^l$. In particular, if
$\pi_\mu$ is not $\chi_l$-spherical, then $P_{\mu,l}(V_\mu)=0$. Fix an orthonormal
basis of $V_\mu$, say $v_1,\ldots , v_{d(\mu )}$. In case $V_\mu$ is $\chi_l$-spherical, we assume
that $v_1=e_{\mu,l}$. Then 
\begin{eqnarray*}
d(\mu ) \int_K \xi_\mu (uk)\chi_l(k)\, dk&=& \sum_{j=1}^{d (\mu )} 
\ip{v_j}{\int_K\, \chi_l(k^{-1})\pi_\mu (u)\pi_\mu (k)v_j\, d k}\\
&=&\sum_{i=1}^{d(\mu )}\ip{v_j}{\pi_{\mu}(u)P_{\mu ,l}v_j}\\
&=&\left\{\begin{matrix} 0 & \text{ if } & \pi_\mu \text{ is not } \chi_l \text{ spherical}\\
\psi_{\mu_{\fb},l} & \text{ if } & \pi_\mu \text{ is  } \chi_l \text{ spherical}
\end{matrix}\right.
\end{eqnarray*}
where $\mu_{\fb}$ is the projection of $\mu\in i\fh^*$ onto $i\fb^*$.  The claim (\ref{eq:Ql}) now 
follows from our description of 
$\gL_l^+(U)= \{(\mu , 0 , \mu_0)\mid \mu \in \gL^+_l\}$. The claim that $Q_l(f)$ is supported in a ball
of radius $r$ follows as Lemma 9.3 in \cite{os}.
\end{proof}

In our case there are extra factors $\mu_0$ and $\rho_0$ which would come into the picture of restriction of Paley-Wiener 
spaces to subspaces. We use some facts on restriction of Weyl groups to take care of these factors, as in the following lemmas.

Recall that $\rho_0 = \rho_\fh |_{\mathfrak{a}^\perp}$ and $\rho_\fh = \rho + \rho_0$. The map 
\[F  \longmapsto  \Psi_\mathfrak{h}\, (F),\; \Psi_\mathfrak{h}\, (F)\, (\lambda) = F\, (\lambda - \rho_\fh)\]
is an isomorphism from $\mathrm{PW}_r\, (\mathfrak{h}_{\mathbb{C}}^\ast)^{R\, (W_\fh)}$ onto 
$\mathrm{PW}_r\, (\mathfrak{h}_{\mathbb{C}}^\ast)^{W_\fh}$
with the inverse $\Psi_\mathfrak{h}^{-1}\, (F)\, (\lambda) = F\, (\lambda + \rho_\fh)$.
As mentioned earlier, we simply write $(0,\, 0,\, \mu_0)$ as $\mu_0$. The map 
\[F \longmapsto \Phi\, (F),\quad \Phi\, (F)\, (\lambda) = F\, (\lambda + \mu_0 + \rho_0)\]
is an isomorphism of $\mathrm{PW}_r(\fh_\C^*)$ onto itself.  

Recall that $\Delta_0 = \{\alpha \in \Delta\, \mid\, \alpha |_\fa = 0\}$. Let 
\begin{eqnarray*}
\widetilde{W} & = & \{w \in W_\fh\, \mid\, w\, (\fa_\C) = \fa_\C\}\\
W_0 & = & \{w \in W_\fh\, \mid\, w |_\fa = \text{Id}\}.
\end{eqnarray*}
We have $\widetilde{W} \subseteq W_\fh$ and $W_0 \subseteq \widetilde{W}$.  

\begin{lemma}
\label{lem:mu0}
If $w \in \widetilde{W}$, then $w\, \mu_0 = \mu_0$.
\end{lemma}

\begin{proof}
Since $\mu_0 = \lambda\, (i\, X)$ (see Section \ref{s:sphRep}), it remains to show that 
if $w \in \widetilde{W}$ then $w\, X = X$. Recall \cite[(4.4)]{sc} for the definition of $X$. Note that 
$\widetilde{W} |_\mathfrak{a} = W$, and elements of $W$ are permutations and sign changes. The permutations of $\eps_j$'s
are given by products of reflections $r_{\frac{1}{2}\, (\eps_i - \eps_j)}$'s. Let $\beta \in \Delta$ be given by 
$\beta = \beta^+ + \beta^-$ where $\beta^- = \frac{1}{2}\, (\eps_i - \eps_j) \in \fa^\ast$ and $\beta^+ \in (\fa^\perp)^\ast$.
Let $\overline{\beta} = -\beta^+ + \beta^-$. Then $\langle \beta, \overline{\beta} \rangle = 0$, 
$r_\beta\, r_{\overline{\beta}} \in \widetilde{W}$, and 
\[r_\beta\, r_{\overline{\beta}} |_\mathfrak{a} = r_{\frac{1}{2}\, (\eps_i - \eps_j)}.\]
We have $(r_\beta\, r_{\overline{\beta}})\, X = X$. 
It is easy to see that any $w \in \widetilde{W}$ with $w |_\fa = r_{\frac{1}{2}\, (\eps_i - \eps_j)}$
also satisfies $w\, X = X$.
On the other hand, the sign changes of $\eps_j$'s are given by products of $r_{\eps_j}$'s. 
But then $r_{\eps_j}\, X = X$ since $X \perp \fa$. This is correct for any element in $\widetilde{W}$ whose restriction
on $\mathfrak{a}$ is a sign change. Therefore, $w\, X = X$ for any $w \in \widetilde{W}$.
\end{proof}

\begin{lemma}
\label{lem:inv}
Let $G \in \mathrm{PW}_r\, (\mathfrak{h}_\mathbb{C}^\ast)^{W_\fh}$. Then 
$\Phi (G)|_{\mathfrak{a}_{\mathbb{C}}^*} = G\, (\, \cdot\, + \mu_0 + \rho_0)$ is $W$-invariant, i.e. 
if $w \in W$ and $\mu \in \mathfrak{a}_\mathbb{C}^\ast$, then
\[G\, (w\, \mu + \mu_0 + \rho_0) = G\, (\mu + \mu_0 + \rho_0).\]
\end{lemma}

\begin{proof}	
Let $w \in W$. Let $\widetilde{w} \in \widetilde{W}$ be such that $\widetilde{w}\, |_\fa = w$.
Since $G$ is $W_\fh$-invariant, it is thus $\widetilde{W}$-invariant. In view of Lemma \ref{lem:mu0} we get
\[G\, (w\, \mu + \mu_0 + \rho_0) = G\, (\mu + \widetilde{w}^{-1}\, \mu_0 + \widetilde{w}^{-1}\, \rho_0) 
= G\, (\mu + \mu_0 + \widetilde{w}^{-1}\, \rho_0).\]
Note that $\widetilde{w}^{-1}\, \Delta_0^+ = \Delta_0^+$, and we can choose $w_0 \in W_0$ such that 
$w_0\, (\widetilde{w}^{-1}\, \Delta_0^+) = \Delta_0^+$, in particular, choose $w_0$ s.t. 
$w_0\, \widetilde{w}^{-1}\, (\rho_0) = \rho_0$. 
Moreover, $w_0\, \mu = \mu$. It follows that 
\[G\, (w\, \mu + \mu_0 + \rho_0) = G\, (w_0\, (\mu + \mu_0 + \widetilde{w}^{-1}\, \rho_0)) 
= G\, (\mu + \mu_0 + \rho_0).\]
\end{proof}
 
Let $G \in \PW_r(\fh_\C^\ast)$. Let $k = |W_\fh|$. Let $P_1,\ldots ,P_k$ be a basis for $S(\fh)$ over
$S(\fh)^{W_\fh}$. Here, $S\, (\fh)$ is the symmetric algebra of $\C$-valued polynomials on $\fh_\C^\ast$, and 
$S(\fh)^{W_\fh}$ consists of $W_\fh$-invariant elements in $S\, (\fh)$.
According to Rais \cite{R83}, published proof due
to L. Clozel and P. Delorme, \cite{CD90},  there exists
$G_1,\ldots ,G_k\in \PW_r(\fh_\C^\ast)^{W_\fh}$ 
such that
\[G=P_1G_1+\ldots + P_kG_k\, .\] 
We are now ready to prove that the $\chi_l$-spherical Fourier transform $\cS_l$ (\ref{eq22}) is surjective.

\begin{proof}(Part $(2)$ of Theorem \ref{thm1})
Let $F\in \PW_r(\fb_\C^\ast)^{R(W)}$. Then $\Psi\, (F) \in \text{PW}_r\, (\mathfrak{b}_{\mathbb{C}}^\ast)^W$.
It follows from Cowling \cite{C86} that there exists $E\in \mathrm{PW}_r(\mathfrak{h}_{\mathbb{C}}^\ast)$ 
such that $\Phi (E)|_{\mathfrak{a}_{\mathbb{C}}^*}=\Psi(F)\, (\, \cdot\, + \mu_0)$, i.e. 
\[E\, (\mu + \mu_0 + \rho_0) = \Phi (E)|_{\mathfrak{a}_{\mathbb{C}}^*}\, (\mu) = \Psi\, (F)\, (\mu + \mu_0),\; \mu \in \mathfrak{a}_{\mathbb{C}}^\ast.\] 
By the above results of Rais, there exist polynomials $P_j\in S(\fh)$
and $G_j\in \PW_r(\fh_\C^\ast)^{W_\fh}$ such that $E = \sum_{j=1}^k\, P_j\, G_j$. Hence
\begin{equation}\label{eq:Psi}
\Psi (F)\, (\, \cdot \, + \mu_0) =\sum_{j=1}^k  \Phi (P_j)|_{\fa_\C^*}\; \Phi (G_j)|_{\fa_\C^*}\, .
\end{equation}
Taking the average of (\ref{eq:Psi}) over $W$ gives that
\begin{eqnarray*}
\Psi\, (F)\, (\mu + \mu_0) & = & \sum_{j=1}^k\, \left(\frac{1}{|W|}\, \sum_{w \in W}\, \Phi (P_j)|_{\mathfrak{a}_{\mathbb{C}}^*}\, 
(w\, \mu)\, \Phi (G_j)|_{\mathfrak{a}_{\mathbb{C}}^*}\, (w\, \mu)\right)\\
& = & \sum_{j=1}^k\, \underbrace{\left(\frac{1}{|W|}\, \sum_{w \in W}\, \Phi (P_j)|_{\mathfrak{a}_{\mathbb{C}}^*}\, 
(w\, \mu)\right)}_{=:\; q_j\, (\mu)}\, \Phi (G_j)|_{\mathfrak{a}_{\mathbb{C}}^*}\, (\mu)
\end{eqnarray*}
where by Lemma \ref{lem:inv}
\[\Phi (G_j)|_{\mathfrak{a}_{\mathbb{C}}^*}\, (w\, \mu) = G_j\, (w\, \mu + \mu_0 + \rho_0) = G_j\, (\mu + \mu_0 + \rho_0)
= \Phi (G_j)|_{\mathfrak{a}_{\mathbb{C}}^*}\, (\mu),\; w \in W.\]
Note that $q_j \in S\, (\fa)^W$. Let $D_j\in \cD_l(\bY)$ be such that $q_j(\lambda )=\overline{\zeta_l(D^*_j,\lambda)}$,
$\lambda \in \fa_\C^\ast$, see the discussion at the beginning of Section \ref{s:holExt}. 

By the Paley-Wiener theorem for $C_r^\infty (U)^U$, there exists $\varphi_j\in C_r^\infty (U)^U$
with spherical Fourier transform $\Psi_{\fh}^{-1}G_j$. Then $f_j=Q_l(\varphi_j)\in C_r^\infty (U// K ,\cL_l)$
has the $\chi_l$-spherical Fourier transform:
\[\mathcal{S}_l\, (f_j)\, (\mu + \mu_0) = \Phi (G_j)|_{\mathfrak{a}_{\mathbb{C}}^*}\, (\mu + \rho).\]
It follows that $F$ is the $\chi_l$-spherical Fourier transform of $f:=D_1f_1+\ldots + D_k\, f_k$. 
The surjectivity now follows from the
fact that differentiation does not increase supports and hence $f\in C_r^\infty (U// K; \cL_l)$.
\end{proof}

\appendix

\section{Estimates for the Heckman-Opdam Hypergeometric Functions}
\label{hyper}

\noindent
We first review some facts on the theory of Heckman-Opdam hypergeometric functions. 
We refer to \cite{hs} for notations and basic definitions. 
We do not assume that $\Sigma$ corresponds to a Hermitian symmetric space. 
Recall that a multiplicity function $m:\Sigma \to \C$ is a $W$-invariant function. It is said
to be positive if $m(\alpha )\ge 0$ for all $\alpha$. The set of multiplicity functions is denoted by $\cM$ and the subset
of positive multiplicity functions is denoted by $\cM^+$. The Harish-Chandra series corresponding to
a multiplicity function $m$ is denoted by $\Phi (\lambda, m; a)$ and
the $c$-function is denoted by $c(\lambda, m)$.

\begin{definition}
\label{def:hg}
The function
$$F (\lambda, m; a) = \sum_{w \in W} c (w \lambda, m) \Phi (w \lambda, m; a)$$
is called the \textit{hypergeometric function} on $A$ associated with the triple $(\g{a},\, \Sigma,\,  m)$. 
\end{definition}

\begin{theorem} Let
\[\cM_{\ge} =\{m\in\cM\mid (\forall \alpha\in\Sigma_\ast)\, \, m_\alpha + m_{\alpha/2} \ge 0,\, m_\alpha \ge 0\}\, .\]
Then the following hold:
\begin{itemize}
\item[(1)] There exists an open set $\cM_{\text{reg}}$ containing $\cM_{\ge}$ and an open set
$\mathcal{V} \subset A_\C$ containing $A$ such that $F(\lambda , m; a)$ is holomorphic on
$\fa^*_\C\times \cM_{\text{reg}}\times \mathcal{V}$, and satisfies 
\begin{eqnarray*}
F\, (w\, \lambda, m;\, a) = F\, (\lambda, m;\, a),\; \forall\, w \in W\\
F\, (\lambda, m;\, w\, a) = F\, (\lambda, m;\, a),\; \forall\, w \in W
\end{eqnarray*}
with $(\lambda, m ; a) \in \fa^*_\C\times \cM_{\text{reg}}\times \mathcal{V}$.
\item[(2)] One can take $\mathcal{V}$ in (1) to be $\exp\, (\fa+\Omega)$, where $\Omega$ is defined as in (\ref{eq:Omega}).
\end{itemize}
\end{theorem}

\begin{proof} For (1) see \cite[Theorem 4.4.2 and Remark 4.4.3]{hs}. See also \cite[Theorem 3.15]{op}. 
For (2) see Remark 3.17 in \cite{bop}.
\end{proof}

\begin{proposition} Let the notation be as above. Let $l \in \Z$ and let $\eta_l^\pm$ be as in (\ref{eq:eta}).
Then the following hold:
\begin{itemize}
\item[(1)] The multiplicity functions $m_{\pm} (l)$ are in $\cM_{\ge}$.
\item[(2)] For $\lambda \in \g{a}_{\C}^\ast$,
\[\displaystyle \varphi_{\lambda,\, l}\, |_A = \eta_l^\pm\, F\, (\lambda,\, m_\pm(l);\; \cdot)\]
where the $\pm$ sign indicates that both possibilities are valid. 
\end{itemize}
\end{proposition}

\begin{proof} (1) follows from the definition of $m_\pm (l)$ (cf. (\ref{mplus}) and (\ref{mminus})), and 
(2) is \cite[p.76, Theorem $5.2.2$]{hs}.
\end{proof}

\begin{remark}
\label{eta:estim}
{\rm For $X \in \Omega$ and $\lambda \in \g{a}_{\C}^\ast$ we have 
\[\varphi_{\lambda,\, l} (\exp X) = \eta_l  (\exp X) F (\lambda, m (l); \exp X).\] 
Since $\alpha (\g{b}) \subset i  \R$ for $\alpha \in \Sigma$,
$$0 < |\eta_l (\exp X)| = \prod_{\alpha \in \mathcal{O}_s^+} \left|\cos\, \im\, \alpha\, (X)\right|^{2|l|} \leq 1.$$
Thus $\eta_l$ is holomorphic on $A (\exp \Omega)$, $W$-invariant on $A (\exp \Omega)$, and bounded on $\exp \Omega$.}
\end{remark}
 
In the main part showing that the $\chi_l$-spherical Fourier transform maps the space
$C_r^\infty (U//K;\cL_l)$ into the Paley-Wiener space, a good control over
the growth of the hypergeometric functions was needed. Proposition $6.1$ in \cite{op} gives a
uniform estimate both in $\lambda \in \g{a}_{\C}^\ast$ and in $Z \in \g{a} + \overline{\Omega/2}$
(recall the difference in the notation) in case all multiplicities are positive. 
But we need similar estimates where some multiplicities are allowed
to be negative. In the following
we will generalize Opdam's results to multiplicities in $\cM_{\ge}$. We
also point out, that Opdam's estimates holds for $Y\in  \overline{\frac{1}{2}\Omega}$,
but our estimates, with possibly a different constant in the exponential
growth, holds for $Y\in \Omega$. 
Our proof is based on ideas from \cite{op} but uses a
different regrouping of terms as we will point out later.

For $m \in \mathcal{M}_{\text{reg}}$ let
$$\wt{\rho} = \wt{\rho} (m) = \frac{1}{2} \sum_{\alpha \in \Sigma^+} |m_\alpha| \alpha.$$
 
\begin{proposition}
\label{pro1}
Let $m \in \mathcal{M}_{\ge}$. Let $F$ be the hypergeometric function associated with $\Sigma$ and $m$.
Let $\eps > 0$. Then there is a constant $C = C_\eps > 0$
depending on $\eps$ such that
$$|F\, (\lambda,m; \exp\, Z)| \leq |W|^{\frac{1}{2}}\, \exp\, (- \min_{w \in W} \im\, (w \lambda\, (Y)) + \frac{C}{2}\, \max_{w \in
W}\, w\, \wt{\rho}\, (Y) + \max_{w \in W}\, \re (w \lambda\, (X)))$$
where $Z = X + i\, Y$ with $X,\, Y \in \g{a}$ and $|\alpha\, (Y)| \leq \pi - \eps$ for all $\alpha \in \Sigma$.
\end{proposition}

\begin{proof}
Let $\phi_w (\exp Z) = G (\lambda, m,w^{-1} Z)$ where $G$ is the nonsymmetric hypergeometric function defined as in \cite[Theorem
$3.15$]{op}, so that 
\[F (\lambda,\, m;\, \exp Z) = |W|^{-1} \sum_{w \in W}\, G (\lambda, m, w^{-1}Z)\, .\] 
In the following we will often write $\phi_w$ instead of
$\phi_w(\exp Z)$. By Definition $3.1$
and Lemma $3.2$ in \cite{op} we have

\[\partial_\xi \phi_w = - \frac{1}{2} \sum_{\alpha \in \Sigma^+} m_\alpha \alpha (\xi) \left[\frac{1 + e^{-2 \alpha
(Z)}}{1 - e^{-2 \alpha (Z)}}\, (\phi_w - \phi_{r_\alpha w}) - \mathrm{sgn} (w^{-1} \alpha) \phi_{r_\alpha w}\right]
+(w \lambda,\xi) \phi_w\, .\]
Here, $\mathrm{sgn}\, (\alpha) = 1$ if $\alpha \in \Sigma^+$, and $\mathrm{sgn}\, (\alpha) = -1$ if $\alpha \in -\Sigma^+$.
We get by taking complex conjugates,
\[\partial_\xi \overline{\phi}_w = - \frac{1}{2}  \sum_{\alpha \in \Sigma^+} m_\alpha \alpha (\overline{\xi})\left[\frac{1
+ e^{-2 \alpha (\overline{Z})}}{1 - e^{-2 \alpha (\overline{Z})}} (\overline{\phi}_w - \overline{\phi}_{r_\alpha w}) -
\mathrm{sgn} (w^{-1} \alpha) \overline{\phi}_{r_\alpha w}\right] 
+ (w \overline{\lambda},\overline{\xi}) \overline{\phi}_w\, .\]
It follows that
\begin{eqnarray*}
& & \partial_\xi\, \sum_w\, |\phi_w|^2 \\
& = & \sum_w [(\partial_\xi \phi_w) \overline{\phi}_w + \phi_w (\partial_\xi \overline{\phi}_w)]\\
& = & - \frac{1}{2}\, \sum_{\alpha \in \Sigma^+, w} [m_\alpha\, \alpha (\xi)
 \left(\frac{1 + e^{-2 \alpha\,
(Z)}}{1 - e^{-2 \alpha\, (Z)}} (\phi_w - \phi_{r_\alpha w}) \overline{\phi}_w - \mathrm{sgn} (w^{-1}\, \alpha)
\phi_{r_\alpha\, w} \overline{\phi}_w\right)\\
& & + m_\alpha \alpha (\overline{\xi}) \left(\frac{1 + e^{-2 \alpha (\overline{Z})}}{1 - e^{-2 \alpha (\overline{Z})}}
(\overline{\phi}_w - \overline{\phi}_{r_\alpha w}) \phi_w - \mathrm{sgn} (w^{-1} \alpha) \overline{\phi}_{r_\alpha\,
w} \phi_w\right)]\\
& & + 2 \sum_w \re (w \lambda (\xi)) |\phi_w|^2.
\end{eqnarray*}
For fixed $\alpha$, we add the terms with index $w$ and $r_\alpha w$. Then
\begin{eqnarray*}
& & \partial_\xi \sum_w |\phi_w|^2 \\
& = & - \frac{1}{4}  \sum_{\alpha \in \Sigma^+,  w} m_\alpha \left[\alpha (\xi) \frac{1 + e^{-2 \alpha (Z)}}{1 - e^{-2
\alpha (Z)}} + \alpha (\overline{\xi}) \frac{1 + e^{-2 \alpha (\overline{Z})}}{1 - e^{-2 \alpha
(\overline{Z})}}\right] |\phi_w - \phi_{r_\alpha w}|^2\\
& & + \sum_{\alpha \in \Sigma^+, w} m_\alpha \mathrm{sgn} (w^{-1} \alpha) \im (\alpha (\xi)) \im
(\overline{\phi}_w \phi_{r_\alpha w}) + 2 \sum_w \re (w \lambda (\xi)) |\phi_w|^2.
\end{eqnarray*}
Observe that
$$|1 - e^{- 2 \alpha (Z)}|^2 = (1 - e^{- 2 \alpha (Z)}) \overline{1 - e^{- 2 \alpha (Z)}} = (1 - e^{- 2 \alpha (Z)})
(1 - e^{- 2 \alpha (\overline{Z})}),$$
which gives
\begin{eqnarray*}
& & \alpha (\xi) \frac{1 + e^{-2 \alpha (Z)}}{1 - e^{-2 \alpha (Z)}} + \alpha (\overline{\xi}) \frac{1 + e^{-2
\alpha (\overline{Z})}}{1 - e^{-2 \alpha (\overline{Z})}}\\
& = & \frac{\alpha (\xi) (1 + e^{- 2 \alpha (Z)}) (1 - e^{- 2 \alpha (\overline{Z})}) + \alpha (\overline{\xi}) (1 +
e^{-2 \alpha (\overline{Z})}) (1 - e^{-2 \alpha (Z)})}{|1 - e^{- 2 \alpha (Z)}|^2}.
\end{eqnarray*}
Let $\eps > 0$ and write $Z = X + i Y$ with $X, Y \in \g{a}$ and $|\alpha (Y)| \leq \pi - \eps$, for all $\alpha \in \Sigma$.
Let $\alpha (X) = t \in \R$ and $\alpha (Y) = s \in \R$. Then $\alpha (Z) = t + i s$. We have
\[
(1 + e^{- 2 \alpha (Z)}) (1 - e^{- 2 \alpha (\overline{Z})}) =   1 - e^{-4 t} - 2 i e^{- 2 t} \sin (2 s)\, .\] 
Similarly,
$$(1 + e^{-2 \alpha (\overline{Z})}) (1 - e^{-2 \alpha (Z)}) = 1 - e^{-4 t} + 2 i e^{- 2 t} \sin (2 s)\, .$$
A simple calculation the shows that 
\begin{align*}
& \alpha (\xi) (1 + e^{- 2 \alpha (Z)}) (1 - e^{- 2 \alpha (\overline{Z})}) + \alpha (\overline{\xi}) 
(1 + e^{-2 \alpha (\overline{Z})}) (1 - e^{-2 \alpha (Z)})\\ 
=\; & 2 \re (\alpha (\xi)) (1 - e^{-4 \alpha (X)}) + 4 \im (\alpha (\xi)) e^{-2 \alpha (X)} \sin (2\alpha (Y)).
\end{align*}
Hence,
\begin{align}
& \partial_\xi \sum_w |\phi_w|^2 \label{eq37}\\
= & - \frac{1}{2} \sum_{\alpha \in \Sigma^+, w} m_\alpha \left[\frac{\re (\alpha (\xi)) (1 - e^{-4 \alpha (X)}) +
\frac{2 \im (\alpha (\xi))  \sin (2\alpha (Y))}{e^{2 \alpha (X)}}}{|1 - e^{- 2 \alpha (Z)}|^2}\right] |\phi_w -
\phi_{r_\alpha w}|^2 \nonumber\\
& + \sum_{\alpha \in \Sigma^+, w} m_\alpha \mathrm{sgn} (w^{-1} \alpha) \im (\alpha (\xi)) \im
(\overline{\phi}_w \phi_{r_\alpha w}) + 2 \sum_w \re (w \lambda (\xi)) |\phi_w|^2. \nonumber
\end{align}
We first take $X, \xi \in \g{a}^{\mathrm{reg}}$ such that they are in the same Weyl chamber. Let $\mu \in \{w \re \lambda\}_{w
\in W}$ be such that $\mu (\xi) = \max_w \re (w \lambda) (\xi)$. Then $(w \re \lambda - \mu) (\xi) \leq 0$. The
formula (\ref{eq37}) gives
\begin{eqnarray}
& & \partial_\xi (e^{- 2 \mu (X)}\, \sum_{w \in W}\, |\phi_w (\exp\, Z)|^2) \nonumber\\
& = & - \frac{1}{2} \sum_{\alpha \in \Sigma^+, w} m_\alpha \frac{\alpha (\xi) (1 - e^{- 4 \alpha (X)})}{|1 - e^{- 2
\alpha (Z)}|^2} |\phi_w - \phi_{r_\alpha w}|^2 e^{- 2 \mu (X)} \label{eq30}\\
& & + 2 \sum_{w \in W} (w \re \lambda - \mu) (\xi) |\phi_w|^2 e^{- 2 \mu (X)} \label{eq38},
\end{eqnarray}
Observe that the term (\ref{eq38}) is clearly nonpositive. In the term (\ref{eq30}), the factor 
\[|\phi_w - \phi_{r_\alpha w}|^2 e^{- 2 \mu (X)} \geq 0.\]
We let $m_{\alpha/2} = 0$ if $\alpha/2$ is not a root. Consider
\begin{eqnarray}
& & \sum_{\alpha \in \Sigma^+,\, w} m_\alpha \frac{\alpha (\xi) (1 - e^{- 4 \alpha (X)})}{|1 - e^{- 2 \alpha
(Z)}|^2}\, |\phi_w - \phi_{r_\alpha w}|^2 e^{- 2 \mu (X)} \nonumber\\
& = & \sum_{\alpha \in \Sigma_\ast^+,\, w} \left[m_\alpha \frac{\alpha (\xi) (1 - e^{-4 \alpha (X)})}{|1 - e^{- 2\alpha (Z)}|^2}
+ m_{\alpha/2} \frac{\frac{1}{2}\, \alpha (\xi) (1 - e^{-2\alpha (X)})}{|1 - e^{- \alpha (Z)}|^2}\right]\,
|\phi_w - \phi_{r_\alpha w}|^2 e^{- 2 \mu (X)} \nonumber\\
& = & \sum_{\alpha \in \Sigma_\ast^+,\, w} \alpha (\xi) \frac{1 - e^{- 2\alpha (X)}}{|1 - e^{- \alpha (Z)}|^2}
\left[m_\alpha \frac{1 + e^{- 2\alpha (X)}}{|1 + e^{- \alpha (Z)}|^2} + \frac{1}{2} m_{\alpha / 2}\right]\,
|\phi_w - \phi_{r_\alpha w}|^2 e^{- 2 \mu (X)}. \label{eq39}
\end{eqnarray}
Since $X, \xi$ are in the same Weyl chamber, $\alpha (\xi) (1 - e^{-2 \alpha (X)}) \geq 0$ for all $\alpha \in \Sigma^+$. 
Since $m_{\alpha / 2} \geq - m_\alpha$ and $m_\alpha \geq 0$ for all $\alpha \in \Sigma_\ast^+$, then
$$m_\alpha \frac{1 + e^{- 2\alpha (X)}}{|1 + e^{- \alpha (Z)}|^2} + \frac{1}{2} m_{\alpha / 2} \geq m_\alpha \frac{1 +
e^{- 2\alpha (X)}}{|1 + e^{- \alpha (Z)}|^2} - \frac{1}{2} m_\alpha = m_\alpha
\left[\frac{1 + e^{- 2\alpha (X)}}{|1 + e^{- \alpha (Z)}|^2} - \frac{1}{2}\right] \geq 0.$$
The reason is as follows:
$$\frac{1 + e^{- 2\alpha (X)}}{|1 + e^{- \alpha (Z)}|^2} - \frac{1}{2} = \frac{2 (1 + e^{- 2 t}) - |1 + e^{-t} e^{-i
s}|^2}{2 |1 + e^{-t} e^{-i s}|^2} \geq 0$$
if and only if the numerator is nonnegative, which is clearly as
\begin{eqnarray*}
2 (1 + e^{- 2t}) - |1 + e^{-t} e^{-i s}|^2 & = & 1 + e^{- 2t} - 2 e^{-t} \cos (s/2)\\
& \geq & 1 + e^{- 2 t} - 2 e^{-t} \\
& = & (1 - e^{- t})^2 \geq 0.
\end{eqnarray*}
It follows that (\ref{eq39}) is nonnegative. Thus the term (\ref{eq30}) is nonpositive and hence
$$\partial_\xi (e^{- 2 \mu (X)} \sum_{w \in W} |\phi_w (\exp\, Z)|^2) \leq 0.$$
This implies 
\begin{eqnarray*}
e^{- 2 \max_w \re (w \lambda (X))} \sum_w |\phi_w (\exp\, Z)|^2 & \leq & e^{- 2 \max_w \re (w \lambda (0))} \sum_w
|\phi_w (\exp\, (0 + i Y))|^2\\
& = & \sum_w |\phi_w (\exp\, (i Y))|^2
\end{eqnarray*}
if $X \in \g{a}^{\text{reg}}$, and by continuity this holds for all $X \in \g{a}$. Note that
$$|\phi_e (\exp\, Z)| = |G (\lambda, m, e^{-1} Z)| = |G (\lambda, m, Z)|$$
and $|\phi_e (\exp\, Z)|^2 \leq \sum_w |\phi_w (\exp\, Z)|^2$ which implies 
$|\phi_e (\exp\, Z)| \leq (\sum_w |\phi_w (\exp\, Z)|^2)^{1/2}$. Hence, we have
\begin{equation}
\label{eq41}
|G (\lambda, m, X + i Y)| \leq e^{\max_w \re (w \lambda (X))} \left(\sum_w |\phi_w (\exp\, (i Y))|^2\right)^{1/2}.
\end{equation}
Substituting $Y = 0$ yields
$$|G (\lambda, m, X)| \leq |W|^{1/2} e^{\max_w \re (w \lambda (X))},$$
where we use the fact that $G (\lambda, m, 0) = 1$ (cf. \cite[Theorem $3.15$]{op}).

\vskip3mm

Next, we take $Y \in \g{a}^{\mathrm{reg}}$ such that $|\alpha (Y)| \leq \pi - \eps$ for all $\alpha \in \Sigma$,
and $\eta \in \g{a}^{\mathrm{reg}}$ belonging to the same Weyl chamber, and let $\xi = i \eta$. Then 
$$\re (w \lambda (\xi)) = - \im (w \lambda (\eta))\qquad \text{and}\qquad\im (\alpha (\xi)) = \re (\alpha
(\eta)).$$ 
Take $\mu \in \{w \im \lambda\}_{w \in W}$ such that $- \im (w \lambda (\eta)) \leq - \mu (\eta)$ for all $w \in W$.
This is to say, $\mu = \min_w \im (w \lambda)$. Observe that
$$\sum_{\alpha \in \Sigma^+} |m_\alpha| |\alpha (\eta)| \leq \max_w \sum_{\alpha \in \Sigma^+} |m_\alpha| \alpha
(w \eta) = 2 \max_w (w \wt{\rho}, \eta).$$
We have
\begin{eqnarray*}
\left|\sum_{\alpha \in \Sigma^+, w} m_\alpha \mathrm{sgn} (w^{-1} \alpha) \im (\alpha (\xi)) \im
(\overline{\phi_w} \phi_{r_\alpha w})\right| & \leq & \sum_{\alpha \in \Sigma^+, w} |m_\alpha|\, |\alpha (\eta)| |\phi_w|
|\phi_{r_\alpha w}|\\
& \leq & 2 \max_w (w \wt{\rho}, \eta) \sum_w |\phi_w|^2.
\end{eqnarray*}
Choose $\nu \in \{w \wt{\rho}\}_{w \in W}$ such that $(\nu, \eta) = \max_w (w \wt{\rho}, \eta)$. Let 
$C > 2$ be a constant to be determined and let
$$H (i Y) = e^{2 \mu (Y)} e^{- C \nu (Y)} \sum_w |\phi_w (\exp\, (i Y))|^2\,.$$
Now using the formula (\ref{eq37}) we obtain 
\begin{align}
(\partial_\xi H) (&i Y) \nonumber \\
 =&  - \sum_{\alpha \in \Sigma^+, w} m_\alpha \frac{\alpha (\eta) \sin 2 \alpha (Y)}{|1 - e^{- 2 \alpha (i Y)}|^2} 
|\phi_w - \phi_{r_\alpha w}|^2 \cdot e^{(2 \mu - C \nu) (Y)}\tag{I}\\
&  - (C-2) (\nu, \eta) \sum_w |\phi_w|^2 e^{(2 \mu - C \nu) (Y)}\tag{II}\\
& + \left[\sum_{\alpha \in \Sigma^+, w} m_\alpha \mathrm{sgn} (w^{-1} \alpha) \im (\alpha (\xi)) \im
(\overline{\phi_w} \phi_{r_\alpha w}) - 2 (\nu, \eta) \sum_w |\phi_w|^2\right] e^{(2 \mu - C \nu)
(Y)}\tag{III}\\
& + \left[2 \sum_w (\mu - w \im \lambda) (\eta) |\phi_w|^2\right] e^{(2 \mu - C \nu) (Y)}, \tag{IV}
\end{align}
where (III) and (IV) are clearly nonpositive.
In the original proof in \cite{op} $\alpha (\eta)\, \sin 2\alpha (Y)\ge 0$ and $m_\alpha \ge 0$ for all $\alpha \in \Sigma^+$. 
Hence (I) is nonpositive. 
Then the author takes $C=2$ to let (II) vanish. In fact any $C\ge 2$ would be
good enough. Our assumptions allows for $\pi \le |2 \alpha (Y)|\le 2\pi - 2\epsilon$ which 
would imply that $\alpha (\eta)\, \sin 2\alpha (Y)\le 0$ and so (I)$\ge 0$. Similar
problems arise since $m_\alpha$ might be nonpositive for some $\alpha \in \Sigma^+$. Thus the rest of the proof differs from the original proof in \cite{op} by  
grouping (I) and (II) together and then choosing the constant $C$ such that the sum becomes
nonpositive. 

As earlier we set $m_{\alpha/2}=0$ if $\alpha/2$ is not a root. For (I) observe that 
\[\sum_{\alpha \in \Sigma^+,\, w \in W}\, m_\alpha\, \frac{\alpha\, (\eta)\, \sin\, 2\alpha\, (Y)}{|1 - e^{- 2\alpha\, (i\, Y)}|^2}\,
|\phi_w - \phi_{r_\alpha w}|^2 \cdot e^{(2 \mu - C \nu) (Y)}\]
\[=\sum_{\alpha \in \Sigma_\ast^+,\, w}\, \left[m_\alpha\, \frac{\alpha\, (\eta)\, \sin\, 2\alpha\, (Y)}{|1 - e^{- 2 \alpha\, (i\,
Y)}|^2} + \frac{1}{2}\, m_{\alpha/2}\, \frac{\alpha\, (\eta)\, \sin\, (\alpha\, (Y))}{|1 - e^{- \alpha\, (i\, Y)}|^2}\right]\,
|\phi_w - \phi_{r_\alpha w}|^2 \cdot e^{(2 \mu - C \nu) (Y)}.\]
Fix a root $\alpha\in\Sigma_\ast^+$. Let  $s_\alpha=s:=\alpha (Y) $. Then
$$|1 - e^{- 2 \alpha (i Y)}|^2 = 4 \sin^2 (s)\, .$$ 
Using that $m_\alpha +m_{\alpha/2} \ge 0$ and $m_\alpha \ge 0$ for all $\alpha \in \Sigma_\ast^+$ we get: 
\begin{eqnarray*}
m_\alpha \frac{\alpha (\eta)  \sin (2s)}{|1 - e^{- 2is }|^2} +\frac{m_{\alpha/2}}{2}
\frac{\alpha (\eta)  \sin (s)}{|1 - e^{- is }|^2} & \geq & m_\alpha \frac{\alpha (\eta) \sin (2s)}{4\, \sin^2\, (s)}
- \frac{1}{2} m_\alpha \frac{\alpha (\eta) \sin (s)}{4\, \sin^2\, (s/2)} \\ 
& = & \frac{m_\alpha \alpha (\eta)}{4} \left[\frac{\sin (2 s)}{\sin^2 (s)} -
\frac{1}{2} \frac{\sin (s)}{\sin^2 (s/2)}\right]\\
& = & - \frac{m_\alpha \alpha (\eta)}{4}\tan (s/2),
\end{eqnarray*}
where we use the formulas $\sin (2 s) = 2 \sin (s) \cos (s)$ and $\cos (2 s) = \cos^2 (s) - \sin^2 (s)$.
Since $\eta, Y$ belong to the same chamber, we see that $\alpha (\eta) \tan (s/2) \geq 0$. 
Since $|s| = |\alpha (Y)| \leq \pi - \eps$, then there is a constant $C_1 = C_1 (\eps) > 0$ depending on $\eps$ such that 
\[\max_{\alpha\in \Sigma_{\ast}^+}|\tan (s_\alpha /2)| \leq C_1\, .\]
Also,
\[\sum_{w \in W}\, |\phi_w - \phi_{r_\alpha\, w}|^2 \leq \sum_{w \in W}\, (|\phi_w| + |\phi_{r_\alpha\, w}|)^2 
= \sum_{w \in W}\, 4\, |\phi_w|^2.\]
It follows that  
\begin{eqnarray*}
\mathrm{(I)} & \leq & C_1 \sum_{\alpha \in \Sigma_\ast^+, \, w} m_\alpha |\alpha (\eta)| |\phi_w|^2
e^{(2 \mu - C \nu) (Y)}\\
& \leq & 2 C_1 \max_w (w \wt{\rho}, \eta) \sum_w |\phi_w|^2e^{(2 \mu - C\nu) (Y)}.
\end{eqnarray*}
Hence (I)+(II)$\leq 0$ if we take  $C \ge 2 + 2 C_1$ ($C$ thus depends on $\eps$). 
It follows that $(\partial_\xi\, H)\, (i\, Y) \leq 0$. Therefore, $H\, (i\, Y) \leq H\, (0) = |W|$. So 
$$\sum_w\, |\phi_w\, (\exp\, (i\, Y))|^2 \leq |W|\, e^{(C\, \nu - 2 \mu)\, (Y)}.$$ 
Together with (\ref{eq41}), we get
$$|G(\lambda,m, Z)| \leq |W|^{\frac{1}{2}} \exp (- \min_{w \in W} \im (w \lambda (Y)) + \frac{C}{2}\max_{w \in
W} w \wt{\rho} (Y) + \max_{w \in W} \re (w \lambda (X))).$$
But note that $|G(\lambda, m,Z)| = |F(\lambda , m;\, \exp Z)|$, we have therefore proved the desired estimate for $F$.
\end{proof}

Since $\wt{\rho}\, (\, \cdot\, )$ (as $\wt{\rho}$ is independent of $\lambda$) and $|W|$ are constants, we restate
Proposition \ref{pro1} as

\begin{proposition}
\label{hypfun:estim}
Let $m \in \mathcal{M}_{\ge}$. 
Let $\eps > 0$. Then there exists a constant $C = C_\eps$ such that 
$$|F(\lambda, m;\exp( X + i Y))| \leq C \exp (\max_{w \in W}\, \re w \lambda (X) - \min_{w \in W} \im
w \lambda (Y)),$$
for all $X \in \Omega_\eps$, $Y \in \g{b}$, and $\lambda \in \g{a}_{\C}^\ast$.
\end{proposition}

\end{document}